\documentclass[12pt]{amsart}
\usepackage[top=30truemm,bottom=30truemm,left=25truemm,right=25truemm]{geometry}
\usepackage{mathrsfs}
\usepackage{tikz-cd}

\usepackage{color}
\usepackage{bm}
\usepackage{amsfonts,amssymb}
\usepackage{dsfont}
\usepackage{amscd}
\usepackage{extarrows}
\usepackage{amsmath}
\usepackage{mathrsfs}
\usepackage{enumerate}
\usepackage{amscd}
\usepackage[all]{xy}
\usepackage[pagebackref,colorlinks]{hyperref}

\theoremstyle{plain} %text of this environment is typesetted in italics
\newtheorem{theorem}{\indent\bf Theorem}[section]

\theoremstyle{definition} %text of this environment is typesetted in roman letters

\newtheorem{example}[theorem]{\indent\bf Example}

\newtheorem{thm}{Theorem}[section]
\newtheorem{cor}[thm]{Corollary}
\newtheorem{lem}[thm]{Lemma}

\theoremstyle{definition}
\newtheorem{defn}{Definition}[section]

\theoremstyle{remark}
\newtheorem{rem}{Remark}[section]

\newcommand{\be}{\begin{equation}}
	\newcommand{\ee}{\end{equation}}
\newcommand{\bea}{\begin{eqnarray}}
	\newcommand{\eea}{\end{eqnarray}}
\newcommand{\ben}{\begin{eqnarray*}}
	\newcommand{\een}{\end{eqnarray*}}
\newcommand{\bt}{\begin{split}}
	\newcommand{\et}{\end{split}}
\newcommand{\bet}{\begin{equation}}
	\newcommand{\mc}{\mathbb{C}}
	\newcommand{\mr}{\mathbb{R}}
	\newcommand{\ra}{\rightarrow}

\newcommand{\bc}{\begin{center}}
\newcommand{\ec}{\end{center}}
\newcommand{\bi}{\begin{itemize}}
\newcommand{\ei}{\end{itemize}}

\begin{document}
		\title[Zero mass conjecture]
{Log truncated threshold and zero mass conjecture}

		\author[F. Deng]{Fusheng Deng}
		\address{Fusheng Deng: \ School of Mathematical Sciences, University of Chinese Academy of Sciences\\ Beijing 100049, P. R. China}
		\email{fshdeng@ucas.ac.cn}
		
	\author[Y. Li]{Yinji Li}
	\address{Yinji Li:  Institute of Mathematics\\Academy of Mathematics and Systems Science\\Chinese Academy of
		Sciences\\Beijing\\100190\\P. R. China}
	\email{liyinji@amss.ac.cn}
	\author[Q. Liu]{Qunhuan Liu}
	\address{Qunhuan Liu: \ School of Mathematical Sciences, University of Chinese Academy of Sciences\\ Beijing 100049, P. R. China}
		\email{liuqunhuan23@mails.ucas.edu.cn}
     \author[Z. Wang]{Zhiwei Wang}
		\address{Zhiwei Wang: Laboratory of Mathematics and Complex Systems (Ministry of Education)\\ School of Mathematical Sciences\\ Beijing Normal University\\ Beijing 100875\\ P. R. China}
		\email{zhiwei@bnu.edu.cn}
	\author[X. Zhou]{Xiangyu Zhou}
		\address{Xiangyu Zhou: Institute of Mathematics\\Academy of Mathematics and Systems Science\\and Hua Loo-Keng Key
			Laboratory of Mathematics\\Chinese Academy of
			Sciences\\Beijing\\100190\\P. R. China}
		%\address{School of
		%	Mathematical Sciences, University of Chinese Academy of Sciences,
		%	Beijing 100049, P. R. China}
		\email{xyzhou@math.ac.cn}

		\thanks{This research is supported by National Key R\&D Program of China (No. 2021YFA1002600, No. 2021YFA1003100),
NSFC grants (No. 12471079, No. 12071035, No. 12288201), and the Fundamental Research Funds for the Central Universities.}
	

		
\begin{abstract}
For plurisubharmonic functions $\varphi$ and $\psi$ lying in the Cegrell class of $\mathbb{B}^n$ and $\mathbb{B}^m$ respectively such that the Lelong number of $\varphi$ at the origin vanishes,
we show that the mass of the origin with respect to the measure $(dd^c\max\{\varphi(z), \psi(Az)\})^n$ on $\mc^n$ is zero for $A\in \mbox{Hom}(\mc^n,\mc^m)=\mc^{nm}$ outside a pluripolar set.
For a plurisubharmonic function $\varphi$ near the origin in $\mc^n$,
we introduce  a new concept coined the \emph{log truncated threshold} of $\varphi$ at $0$ which reflects a singular property of $\varphi$ via a log function near the origin (denoted by $lt(\varphi,0)$)
and derive an optimal estimate of the residual Monge-Amp\`ere mass of $\varphi$ at $0$
in terms of its higher order Lelong numbers $\nu_j(\varphi)$ at $0$ for $1\leq j\leq n-1$, in the case that  $lt(\varphi,0)<\infty$.
These results provide a new approach to the zero mass conjecture of Guedj and Rashkovskii, and unify and strengthen %, and improve 
well-known results about this conjecture.
\end{abstract}
		\maketitle
		\tableofcontents

\section{Introduction}
The study of singularities of plurisubharmonic functions play a central role in several complex variables and complex geometry.
There are a package of remarkable invariants that encodes properties of singularities of plurisubharmonic functions,
such as  (directional, generalized) Lelong numbers, the complex singularity exponents, the multiplier ideal sheaves, the relative type, and the residual (mixed) Monge-Amp\`ere mass
(see e.g. \cite{Dem12} for a systematic introduction).
These invariants and various relations between them are among the central topics of study in several complex variables (e.g., see \cite{Siu74}, \cite{BGZ}, \cite{BFJ08}, \cite{Dem92}, \cite{DD, DD19}, \cite{Ra01, Ra06, Ra09, Ra13-1, Ra13-2}, and a nice survey \cite{K} about the studies before 2000 and a survey \cite{Z} about recent study).

The present paper aims to study the (higher) Lelong number and the residual (mixed) Monge-Amp\`ere mass of plurisubharmonic functions and the relations between them.
In this direction, a fundamental and famous problem is \emph{the zero mass conjecture} proposed by Guedj and Rashkovskii respectively,
which states that residual  Monge-Amp\`ere mass of a plurisuhbarmonic function with isolated singularity must vanish provided that it has zero Lelong number (see \cite[Question 7]{DGZ16}).
On the other hand, it is known that the converse of the conjecture is true, namely the vanishing of the residual  Monge-Amp\`ere mass
implies the vanishing of the Lelong number.

The zero mass conjecture has been studied by many authors,
and it is known that the answer is positive for some interesting cases (see e.g. \cite{Ra01}\cite{BFJ08}\cite{Wi05}\cite{Guj10}\cite{Ra13-1}\cite{ACP19}\cite{L23}\cite{HLX23}\cite{HLX24}).
In the present work, we use different methods to study this conjecture including introduction and study of a new concept {\it log truncated threshold}.
Our main results can cover and strength almost all known results about this conjecture in the literatures.

%\begin{itemize}
%  \item $\varphi$ is toric, i.e. $\varphi(e^{\i\theta_1}z_1,\cdots, e^{i\theta_n}z_n)=\varphi(z_1,\cdots, z_n)$ for all $\theta_1,\cdots, \theta_n\in\mr $ \cite{Ra01};
%  \item $e^\varphi$ is H\"older continuous \cite{BFJ08}\cite{Ra13-1};
%  \item $\varphi$ satisfies the growth condition $\varphi(z)\geq\gamma\log|z|$ near the origin \cite{Wi05};
%  \item $\varphi$ is circular in the sense that $\varphi(e^{i\theta}z)=\varphi(z),\ \theta\in\mr$, see \cite{L23} for dimension 2 and \cite{HLX23} for general dimension;
%  \item $\varphi$ is uniformly directional Lipschitz continuous (see Definition \ref{def:uniform Lip}) \cite{HLX24}.
%\end{itemize}

We need to consider plurisubharmonic functions whose  Monge-Amp\`ere is well defined.
It is shown in the seminal papers \cite{Ce98},\cite{Ce04} that such plurisubharmonic functions should belong to a special class,
now known as the \emph{Cegrell class} (see Definition \ref{def:Cegrell class}).
The Cegrell class on a bounded hyperconvex domain $\Omega$ is denoted by $\mathcal{E}(\Omega)$. For an extension of Cegrell class to compact K\"ahler manifolds, the reader is refereed to \cite{CGZ}.

If $\varphi$ is plurisubharmonic and negative on some neighborhood $\Omega$ of the origin and is locally bounded on $\Omega\backslash\{0\}$,
then $\varphi\in\mathcal{E}(\Omega)$ and the mass of origin with respect to the Monge-Amp\`ere $(dd^c\varphi)^n$ makes sense and is denoted by $(dd^c\varphi)^n(0)$ or $\int_{\{0\}}(dd^c\varphi)^n$.
In this case, we say that $\varphi$ has isolated singularity at $0$.
In general, if $\varphi_1, \cdots, \varphi_n\in \mathcal{E}(\Omega)$ for some neighborhood $\Omega$ of the origin in $\mc^n$,
then the mixed Monge-Amp\`ere measure $dd^c\varphi_1\wedge\cdots\wedge dd^c\varphi_n$ is a well defined positive measure on $\Omega$ and its mass at the origin will be denoted by
$$\int_{\{0\}}dd^c\varphi_1\wedge\cdots\wedge dd^c\varphi_n.$$

We recall the following
\begin{defn}\label{def:Lelong number higher degree}
For $\varphi\in\mathcal{E}(\Omega)$ for some neighborhood $\Omega$ of the origin in $\mc^n$ and for  $j=1,2,\cdots,n$,
the $j$-th Lelong number of $\varphi$ at $0$ is defined to be
\begin{align*}
\int_{\{0\}}(dd^c\varphi)^j\wedge(dd^c\log|z|)^{n-j}.
\end{align*}
\end{defn}

By definition, $\nu_1(\varphi)$ is just the (ordinary) Lelong number of $\varphi$ at $0$,
and $\nu_n(\varphi)$ is the mass of the origin with respect  to the  Monge-Amp\`ere measure $(dd^c\varphi)^n$ at $0$.
As mentioned above, the main aim of the present work is to study relations between $\nu_1(0)$ and $\nu_n(0)$.
In the process, $\nu_j(0)$ with $2\leq j\leq n-1$ will play the role of a bridge that connecting the Lelong number $\nu_1(\varphi)$ and the residual Monge-Amp\`ere measure $\nu_n(\varphi)$.

For $n\geq 1$ and $r>0$, we denote by $\mathbb{B}^n_r$ the open ball in $\mc^n$ with center at the origin and radius $r$.
The unit ball $\mathbb{B}^n_1$ will be simply denoted by $\mathbb{B}^n$.

Our first main result is the following

\begin{thm}\label{thm:max final version}
Suppose $\varphi\in\mathcal{E}(\mathbb{B}^n)$, $\psi\in\mathcal{E}(\mathbb{B}^m)$, where $n,m\geq 1$.
%\mbox{PSH}(\mathbb{B}^n)\cap L^{\infty}_{\mathrm{loc}}(\mathbb{B}^n\backslash\{0\}),\psi\in \mbox{PSH}(\mathbb{B}^m)\cap L^{\infty}_{\mathrm{loc}}(\mathbb{B}^m\backslash\{0\})$.
If $\nu_1(\varphi)=0$, then for
$$A\in\mathrm{Hom}(\mc^n,\mc^m)\cong \mathbb{C}^{nm}$$
outside a pluripolar set, it holds that
\begin{align*}
\nu_n(\max\{\varphi,A^*\psi\})=0,
\end{align*}
where $A^*\psi$ is the compostion of $\psi$ and $A$ given by
$$A^*\psi(z)=\psi(Az),\ z\in \mathbb{B}^n.$$
\end{thm}

Since Theorem \ref{thm:max final version} lies in the heart of the present work,
we give here a sketch of the main ideal of its argument.
The key insight can be described as ``enlarging the dimension", and the key tools involved are the Teissier type inequality for mixed residual Monge-Amp\`ere measures (Lemma \ref{lem:mix})
and Crofton's formula (Lemma \ref{lem:Crofton}).
For simplicity, we just consider such $\varphi$ and $\psi$ that have isolated singularity at the origin.
From the starting point $\nu_1(\varphi)=0$, one can see from Lemma \ref{lem:mix} that $\nu_j(\varphi)=0$ for $2\leq j\leq n-1$.
The zero mass conjecture claim that $\nu_n(\varphi)=0$, but this can not be seen from Lemma \ref{lem:mix}.
To produce information about $\nu_n(\varphi)$ from this lemma, our strategy is to enlarging the dimension $n$ by considering the function
$$\Phi(z,w):=\max\{\varphi(z),\psi(w)\}$$
defined on $\mathbb{B}^n\times\mathbb{B}^m$.
It is clear that $\nu_1(\Phi)=0$, and hence from Lemma \ref{lem:mix} we have
$$\nu_n(\Phi)=\int_{0}(dd^c\Phi)^n\wedge (dd^c\log|(z,w)|)^m=0.$$
Then using the Crofton formula in Lemma \ref{lem:Crofton} and Fubini's theorem,
we can deduce from the above equality that
$$\nu_n(\max\{\varphi,A^*\psi\})=0$$
for almost all $A\in \mathrm{Hom}(\mc^n,\mc^m)$.
This result, though its argument is simple with the insight of enlarging dimension at hand, is powerful enough
to give (with additional trick) an sharp inequality about the relations between $\nu_1(\varphi)$ and $\nu_n(\varphi)$ (Theorem \ref{thm:estimate})
and cover almost all known results in the literature about the zero mass conjecture.
The argument for the plural polarity part of Theorem \ref{thm:max final version} is technically more involved and we omit it here.

A special case of Theorem \ref{thm:max final version} is $m=n$ and $\psi=\varphi$.
If we could take $A$ to be the identity map, then $\max\{\varphi, A^*\psi\}=\varphi$ and hence we get $\nu_n(\varphi)=0$.
Unfortunately, we can just take $A$ sufficiently close to the identity map.
In this sense, Theorem \ref{thm:max final version} implies that the zero mass conjecture is ``almost" true.

\begin{rem}
A classical result in \cite{Sa76} states that pluripolar set in $\mc^n$ cannot contain a subset of positive measure of a totally real submanifold with maximal real dimension $n$.
As a consequence, the Theorem \ref{thm:n-1} holds for $A\in U(n)$ outside a subset with zero measure.
\end{rem}

With the help of a trick of B\l{}ocki (see Lemma \ref{lem:Blocki formula}), Theorem \ref{thm:max final version} can be strengthened as follows.

\begin{thm}\label{thm:n-1}
Suppose $\varphi,\psi\in\mathcal{E}(\mathbb{B}^n)$.
If $\nu_1(\varphi)=0$, then for $A\in \mathrm{Hom}(\mc^n,\mc^n)\cong \mc^{n^2}$ outside a pluripolar set, it holds that
\begin{align*}
\int_{\{0\}} (dd^c\varphi)\wedge(dd^cA^*\psi)^{n-1}
=\cdots=
\int_{\{0\}} (dd^c\varphi)^{n-1}\wedge (dd^cA^*\psi)=0.
\end{align*}
In particular, for $A\in \mathrm{Hom}(\mc^n,\mc^n)\cong \mc^{n^2}$ outside a pluripolar set, it holds that
\begin{align*}
\int_{\{0\}} (dd^c\varphi)\wedge(dd^cA^*\varphi)^{n-1}
=\cdots=
\int_{\{0\}} (dd^c\varphi)^{n-1}\wedge (dd^cA^*\varphi)=0.
\end{align*}
\end{thm}

We now turn our focus to relationships among $\nu_1(\varphi),\cdots, \nu_n(\varphi)$ for the case that $\nu_1(\varphi)$ is not necessarily vanishing.
For this purpose, we need to pose certain growth condition on the considered plurisubharmonic function by introducing the following concept.

\begin{defn}\label{def:log threshold}
Let $\varphi$ be a plurisubharmonic function defined on a neighborhood of $x\in\mc^n$.
The \emph{log truncated threshold} $lt(\varphi, x)$ of $\varphi$ at $x$ is defined as follows:
\begin{align*}
lt(\varphi,x):=\liminf_{r\rightarrow 0}F(r,x),
\end{align*}
where
\begin{align*}
F(r,x):=\lim_{\varepsilon\rightarrow 0}\sup_{r-\varepsilon<|z-x|<r+\varepsilon}\frac{\varphi(z)}{\log|z-x|}.
\end{align*}
\end{defn}

\begin{rem}
In Definition \ref{def:log threshold}, we can replace the family of ball rings $\{r-\varepsilon<|z-x|<r+\varepsilon\}$ by
the $\varepsilon$-neighborhoods of any family of topologically embedded spheres $S_r$ that surround the origin and shrink to a point as $r\ra 0$.
All the following results involving the log truncated threshold  remain true under this modification of definition.
\end{rem}

If $\varphi$ is a plurisubharonic function defined on some neighborhood of the origin such that $lt(\varphi, 0)<\infty$, it is known that $\varphi$ lies in the Cegrell class of some neighborhood of $0$
and hence $\nu_j(\varphi)$ are defined for $1\leq j\leq n$.
The following theorem gives a relation among $\nu_1(\varphi),\cdots, \nu_n(\varphi)$ in terms of  $lt(\varphi, 0)$ for functions
$\varphi$ with $lt(\varphi, 0)<\infty$.

\begin{thm}\label{thm:estimate}
Suppose $\varphi$ is a plurisubharmonic function on $\mathbb{B}^n$ such that $\gamma:=lt(\varphi,0)<+\infty$,
then for every $\beta=(\beta_1,\cdots,\beta_{n-1})$ such that $\beta_j\geq0$ and $\beta_1+\cdots+\beta_{n-1}=1$, it holds that
\begin{align*}
\nu_n(\varphi)
\leq \big(\frac{\nu_1(\varphi)}{\gamma}\big)^{\beta_1}\cdot\big(\frac{\nu_2(\varphi)}{\gamma^2}\big)^{\beta_2}\cdot...\cdot\big(\frac{\nu_{n-1}(\varphi)}{\gamma^{n-1}}\big)^{\beta_{n-1}}\cdot\gamma^{n}.
\end{align*}
In particular, if $\nu_{1}(\varphi)=0$, it holds that $\nu_n(\varphi)=0$.
\end{thm}

The key new insight in the argument of Theorem \ref{thm:estimate} is considering
the enlarging dimensions $n+m$ for all $m\geq 1$ together and then get the sharp estimate
by taking limit as $m\ra\infty$.
If we just consider one step of enlarging dimension, we can just get an estimate that has certain nonexact constant.

We will show in \S \ref{sec:estimate} by a series of examples that the estimate
given in Theorem \ref{thm:estimate} is optimal.

We now discuss the relations between Theorem \ref{thm:estimate} and some results known in the literatures in connection with the zero mass conjecture.
For a domain $\Omega\subset\mc^n$, we denote by $\mbox{PSH}(\Omega)$ the space of plurisubharmonic functions on $\Omega$.
We can see that the zero mass conjecture holds for a plurisubharmonic function $\varphi\in\mbox{PSH}(\mathbb{B}^n)$ such that $\varphi(z)\geq \gamma \log|z|$
for some $\gamma>0$. This recovers the fundamental main result in \cite{Wi05}.
On the other hand, it is easy to construct functions $\varphi\in\mbox{PSH}(\mathbb{B}^n)$ such that $lt(\varphi)<\infty$,
but there does not exist $\gamma>0$ such that $\varphi(z)\geq \gamma \log|z|$ for all $z\in \mathbb{B}^n$.
As we will see in \S \ref{sec:zero mass conjecture},
Theorem \ref{thm:estimate} improves the results in \cite{HLX23} and \cite{HLX24}.
For an $S^1$-invariant $\varphi\in\mbox{PSH}(\mathbb{B}^n)$,
we denote the maximal directional Lelong number of $\varphi$ by $\lambda_{\varphi}(0)$.
By taking $\beta_1=1,\beta_2=\cdots=\beta_{n-1}=0$, we obtain the following corollary, which strengthens \cite[Theorem 1.2]{HLX23} by removing the dimensional constant $C_n$.
\begin{cor}\label{cor:S1 invariant}
Suppose $\varphi\in\mbox{PSH}(\mathbb{B}^n)\cap L^{\infty}_{\mathrm{loc}}(\mathbb{B}^n\backslash\{0\})$ and $\varphi$ is $S^1$-invariant, it holds that
\begin{align*}
\nu_n(\varphi) \leq\nu_1(\varphi)\lambda_{\varphi}(0)^{n-1}.
\end{align*}
\end{cor}
For a plurisubharmonic function which is uniformly directional Lipschitz continuous near the origin, the directional Lipschitz constant is denoted by $\kappa_{\varphi}(0)$ (see \S \ref{def:uniform Lip} for definition).
The following results strengthens \cite[Theorem 1.1]{HLX24}.

\begin{cor}\label{cor:uniform Lip}
Suppose $\varphi\in\mbox{PSH}(\mathbb{B}^n)$ is uniformly directional Lipschitz continuous near the origin,
then for any $\gamma>\kappa_{\varphi}(0)$ we have $\varphi\geq\gamma\log|z|+O(1)$ near the origin.
Moreover the following estimate for residue mass holds:
\begin{align*}
\nu_n(\varphi)\leq \nu_1(\varphi)\kappa_{\varphi}(0)^{n-1}.
\end{align*}
\end{cor}

The remaining of the article is arranged as follows.
In \S \ref{sec:max isolated sing}, we prove Theorem \ref{thm:max final version} and Theorem \ref{thm:n-1} in the special case that $\varphi$ and $\psi$ have isolated singularity at the origin.
The proofs of the Theorems in their full generality follow essentially the same idea, but are technically much more involved,
hence we set them as an appendix.
The aim of \S \ref{sec:estimate} is to prove Theorem \ref{thm:estimate} and show that the estimate given by it is optimal by a series of examples.
We discuss the relations between our main results and other known results related to the zero mass conjecture and show Corollaries \ref{cor:S1 invariant} and  \ref{cor:uniform Lip}
in \S \ref{sec:zero mass conjecture}.
%In the final \S \ref{sec:example} we construct an example $\varphi\in \mathcal E(\mathbb{B}^n)$
%which does not satisfy the condition in Theorem \ref{thm:estimate},
%but the zero mass conjecture is still true for it.

\subsection*{Acknowledgements}
The fifth author would like to thank Prof. Long Li for his invited talks about his results on the zero mass conjecture at Zhou's seminar.

\section{Vanishing of residual mixed  Monge-Amp\`ere measures}\label{sec:max isolated sing}
In this section, we  give the proof of Theorem \ref{thm:max final version} and Theorem \ref{thm:n-1} in the case that $\varphi$ and $\psi$ have isolated singularity.
The proofs of the Theorems in their full generality follow essentially the same idea, but are technically much more involved, and will be given in the Appendix.

%In this section, we prove Theorem \ref{thm:max final version} for $\varphi,\psi$ with isolated singularities $0$.
\begin{thm}\label{thm:max}
Suppose $\varphi\in\mbox{PSH}(\mathbb{B}^n)\cap L^{\infty}_{\mathrm{loc}}(\mathbb{B}^n\backslash\{0\}),\psi\in \mbox{PSH}(\mathbb{B}^m)\cap L^{\infty}_{\mathrm{loc}}(\mathbb{B}^m\backslash\{0\})$.
If $\nu_1(\varphi)=0$, then for $A\in\mathrm{Hom}(\mc^n,\mc^m)\cong \mathbb{C}^{nm}$ outside a pluripolar set, it holds that
\begin{align*}
\nu_n(\max\{\varphi,A^*\psi\}) =0.
\end{align*}
\end{thm}

For the proof of Theorem \ref{thm:max}, we need two lemmas.
The first lemma is a Teissier type inequality associated to point masses with respect to  mixed Monge-Amp\`ere measures,
which is studied in \cite[Lemma 5.4]{Ce04}(see also \cite[Lemma 2.1]{DH14},\cite[Theorem 1.4]{KR21}, and a global version \cite{D}).

\begin{lem}\label{lem:mix}
Suppose $\varphi_1,\varphi_2 \in \mathcal{E}(\mathbb{B}^n)$ (see Definition \ref{def:Cegrell class} for the definition of $\mathcal{E}(\mathbb{B}^n)$),
%\mbox{PSH}(\mathbb{B}^n)\cap L^{\infty}_{\mathrm{loc}}(\mathbb{B}^n\backslash\{0\})$.
for $j=1,\cdots,n-1$, it holds that
\begin{align*}
     &\Big(\int_{\{0\}}(dd^c\varphi_1)^{j}\wedge(dd^c\varphi_2)^{n-j}\Big)^2 \\
\leq &\Big(\int_{\{0\}}(dd^c\varphi_1)^{j-1}\wedge(dd^c\varphi_2)^{n-j+1} \Big)
\Big( \int_{\{0\}} (dd^c\varphi_1)^{j+1}\wedge(dd^c\varphi_2)^{n-j-1}\Big).
\end{align*}
In particular, if $\nu_{1}(\varphi_1)=0$, it holds that $\nu_j(\varphi_1)=0$ for $j=1,\cdots,n-1$.
\end{lem}

The second lemma we need is known as the Crofton's formula.

\begin{lem}\cite[Chapter III, Cor(7.11)]{Dem}\label{lem:Crofton}
Let $dv$ be the unique $U(n)$-invariant measure with total mass $1$ on the Grassmannian $G(p,n)$ of $p$-dimensional subspaces of $\mc^n$, then we have
\begin{align*}
(dd^c\log|z|)^{n-p}=\int_{S\in G(p,n)}[S]dv(S),
\end{align*}
where for a $p$-dimensional linear subspace $S$ in $\mc^n$, $[S]$ represents the current on $\mc^n$ associated to $S$.
\end{lem}

We now give the proof of Theorem \ref{thm:max}.

\begin{proof}
The proof is divided into several steps.
\begin{itemize}
\item
Step 1. Theorem \ref{thm:max} holds for almost every $A\in\mc^{nm}$.
\end{itemize}

Consider a plurisubharmonic function
$$\Phi(z,w):=\max\{\varphi(z),\psi(w)\}$$
on $\mathbb{B}^n\times\mathbb{B}^m$ with isolated singularity $0$.
From the definition of Lelong number, it is obvious that $\nu_{1}(\Phi)\leq \nu_1(\varphi)=0$,
So Lemma \ref{lem:mix} implies that
\begin{align*}
\int_{\{0\}}(dd^c\Phi)^n\wedge(dd^c\log|(z,w)|)^m=0.
\end{align*}
Crofton's formula yields the following equation
\begin{align*}
 &\int_{\{0\}}(dd^c\Phi)^n\wedge(dd^c\log|(z,w)|)^m\\
=&\lim_{r\ra0}\lim_{i\ra\infty} \int_{\mc^{nm}}\chi_r(dd^c\Phi_i)^n\wedge \int_{S\in G(n,n+m)}[S]dv(S)\\
=&\lim_{r\ra0}\lim_{i\ra\infty} \int_{S\in G(n,n+m)}\Big(\int_{\mc^{nm}}\chi_r(dd^c\Phi_i)^n\wedge[S]\Big)dv(S)\\
=&\int_{S\in G(n,n+m)}\Big(\int_{\{0\}}(dd^c\Phi)^n\wedge[S]\Big)dv(S),
\end{align*}
where $\chi_r$ are smooth cut-off functions decreasing to $\mathds{1}_{\{0\}}$ and $\Phi_i=\Phi*\rho_{\varepsilon_i}$, $\{\rho_{\varepsilon_i}\}$ are standard convolution kernels.

We take basis $\{e_1,\cdots,e_{n+m}\}$ of $\mathbb{C}^{n+m}$ so that
\begin{align*}
&\mathrm{span}\{e_1,\cdots,e_n\}=\mathbb{C}^n\times\{0\}, \\
&\mathrm{span}\{e_{n+1},\cdots,e_{n+m}\}=\{0\}\times\mc^m.
\end{align*}
The neighborhood of $\mc^n \times \{0\}$ in $G(n,n+m)$ can be parametrized by $\mc^{nm}$ as follows
\begin{align*}
\mc^{nm} &\ra G(n,n+m),\\
(a_{jk}) &\mapsto \mathrm{span}\{e_{j}+\sum_{k=1}^ma_{jk}e_{n+k}:j=1,\cdots,n\}.
\end{align*}
For $S=\mathrm{span}\{e_{j}+\sum_{k=1}^ma_{jk}e_{n+k}\}$, if we set $A=(a_{jk})\in\mc^{nm}\cong \mathrm{Hom}(\mc^n,\mc^m)$, it is clear that
\begin{align*}
\int_{\{0\}}(dd^c\Phi)^n\wedge[S]=\int_{\{0\}}(dd^c_z\max\{\varphi(z),\psi(Az)\})^n.
\end{align*}
Since $dv$ is a smooth volume form on $G(n,n+m)$, we can infer that for Lebesgue almost every $A\in\mc^{nm}$,
\begin{align*}
\int_{\{0\}}(dd^c_z\max\{\varphi(z),\psi(Az)\})^n=0.
\end{align*}
\begin{itemize}
\item
Step 2. Theorem \ref{thm:max} holds for $A\in\mc^{nm}$ outside a pluripolar set.
\end{itemize}
The proof is motivated by \cite[Chapter III, p176]{Dem}.
Since the problem is local, we can work on the unit ball $\mathbb{B}^{nm}:=\{A\in\mc^{nm}:\|A\|<1\}$.
The plurisubharmonic function $(z,A)\mapsto\Phi(z,Az)=\max\{\varphi(z),\psi(Az)\}$ is well-defined on $\mathbb{B}^n\times\mathbb{B}^{nm}$.
Without loss of generality, we may assume $\varphi,\psi<-2$.
Glueing $\varphi$ with $M\log|z|$ on $\{\frac{1}{2}<|z|<1\}$, we get $\varphi^{\prime}\in\mbox{PSH}(\mathbb{B}^n)$.
$\varphi^{\prime}(z)=M\log|z|$ in a neighborhood of $\{|z|\geq r_0\}$ with $M\log r_0=-1$ for some $M,r_0>0$.
%By enlarging $M$, we can also assume that $M\log r_0>-\frac{1}{2}$.
% $\varphi$ coincides with $N(|z|^2-1)$ on $\{r_0\leq|z|< 1\}$ with $\varepsilon_0$ small and $N((r_0)^2-1)<-1$.
Considering the half-plane $H:=\{w\in\mc: \mathrm{Re  }w<-2\}$, on $\mathbb{B}^{nm}\times H$ we define
\begin{align*}
(A,w)\mapsto V(A,w):=-\int_{\mathrm{Re }w }^0  \chi'(t) \Big( \int_{\{z\in\mathbb{B}^n: \log|z|<t\}}(dd^c_z\Phi(z,Az))^n\Big) dt ,
\end{align*}
where $\chi\in\mathcal{C}^{\infty}(\mr;\mr)$ is non-decreasing, $\chi(t)=t$ near $\{t\leq \log r_0\}$ and $\chi(t)=0$ near $\{t\geq0\}$.
%Our goal is to show that $V$ is plurisubharmonic on $\mathbb{B}^{nm}\times H$.
Set $\Psi(z,w):=\max\{\log|z|,\mathrm{Re}w\}$, a continuous plurisubharmonic function on $\mathbb{B}^n\times \mathbb{B}^{nm}\times H$.
Fubini theorem then yields that
\begin{align*}
V(A,w)&=-\int_{\mathbb{B}^n}\Big(\int_{\max\{\log|z|,\mathrm{Re}w\}}^0\chi'(t)dt\Big)(dd^c_z\Phi)^n\\
&=\int_{\mathbb{B}^n}\chi(\Psi)(dd^c_z\Phi)^n.
\end{align*}
We approximate $\varphi,\psi$ by $\varphi*\rho_j,\psi*\rho_j$ in a standard way.
Glueing $\varphi*\rho_j$ with $M\log|z|$, we get smooth plurisubharmonic functions $\varphi_j$, which still coincide with $M \log|z|$ in a neighborhood of $\{ |z|\geq r_0 \}$.
Set $\Phi_j:=\max_{\varepsilon_j}\{\varphi_j,\psi*\rho_j\},
\Psi_j :=\max_{ \varepsilon_j}\{\log|z|,\mathrm{Re}w\}$ and define
%Fix $(A,w)$. In a neighborhood of $\{z:1-r_0\leq |z|<1\}$, we have
%\begin{align*}
%\chi(\Psi_j)(dd^c_z\Phi_j)^n=\chi(\Psi)(dd^c_z\Phi_j)^n=\chi(M(|z|^2-1))(dd^cM|z|^2)^n.
%\end{align*}
%In $\{z:1-r_0\leq|z|<1\}$, $\chi(\Psi_j)=\Psi_j$ decrease to $ \chi(\Psi)=\Psi\in\mbox{PSH}(\mathbb{B}^n)\cap L^{\infty}(\mathbb{B}^n)$, $\Phi_j$ decrease to $ \Phi\in\mbox{PSH}(\mathbb{B}^n)\cap L^{\infty}_{\mathrm{loc}}(\mathbb{B}^n\backslash\{0\})$.
a sequence of smooth functions on $\mathbb{B}^{nm}\times H$:
\begin{align*}
(A,w)\mapsto V_j(A,w):=\int_{\mathbb{B}^n}\chi(\Psi_j)(dd^c_z\Phi_j)^n.
\end{align*}
We claim that:
\begin{itemize}
\item[(1)]
$V_{j}$ is plurisubhamonic on  $\mathbb{B}^{nm}\times H$,
\item[(2)]
$V_j$ decreases to $V$ on $\mathbb{B}^{nm}\times H$.
\end{itemize}

To prove $(1)$, we need to show that $dd^c_{(A,w)}V_j\geq0$.
For smooth $(nm,nm)$-form $h$ on $\mathbb{B}^{nm}\times H$ with compact support, we compute
\begin{align*}
\langle dd^c_{(A,w)}V_j,h\rangle
&=\langle V_j,dd^c_{(A,w)}h\rangle \\
&= \int_{\mathbb{B}^{nm}\times H} \Big(\int_{\mathbb{B}^n}\chi(\Psi_{j}) (dd^c_z\Phi_{j})^n \Big) \wedge  dd^c_{(A,w)}h \\
&= \int_{\mathbb{B}^n\times \mathbb{B}^{nm}\times H} \chi(\Psi_{j})(dd^c\Phi_{j})^n\wedge dd^c h \\
&=  \int_{\mathbb{B}^n\times \mathbb{B}^{nm}\times H} dd^c \chi(\Psi_{j}) \wedge (dd^c\Phi_{j})^n \wedge h.
\end{align*}
The third $``="$ is because $dd^c_{(A,w)}h=dd^ch$ does not have terms involving $dz,d\bar{z}$ and has top degree on $(A,w)$-direction, which leads to $dd^c_z$ can be replaced by total differential $dd^c_{(z,A,w)}=dd^c$.
The last $``="$ is because $\chi(\Psi_{j})(dd^c\Phi_{j})^n\wedge h$ has compact support on $\mathbb{B}^n\times\mathbb{B}^{nm}\times H$, Stokes' formula then gives the desired equation.
%The unbounded locus of $\Phi$ is $\{0\}\times \mathbb{B}^{nm}\times H$, which has codimension $n$, thus $(dd^c\Phi)^n$ is well-defined and the last $``="$ follows easily.
Notice that in a neighborhood of $\{|z|\geq r_0\} \times \mathbb{B}^{nm}\times H$,
\begin{align*}
\Phi_j=M\log|z|,\Psi_j=\log|z|,
\end{align*}
so we may infer that
\begin{align*}
&\int_{\{ |z|\geq r_0\}\times \mathbb{B}^{nm}\times H}dd^c \chi(\Psi_j) \wedge(dd^c\Phi_j)^n\wedge h\\
=&\int_{\{ |z|\geq r_0\}\times \mathbb{B}^{nm}\times H} dd^c \chi(\log|z|) \wedge(dd^cM\log|z|)^n\wedge h\\
=&\ 0.
\end{align*}
The last equality is because $dd^c \chi(\log|z|) \wedge(dd^cM\log|z|)^n\equiv0$.
Using the fact that $\chi(\Psi_j)=\Psi_j$ on $\{|z|<r_0\} \times \mathbb{B}^{nm}\times H$, we discover
\begin{align*}
\langle dd^c_{(A,w)}V_j,h\rangle=\int_{\{|z|<r_0\} \times \mathbb{B}^{nm}\times H}dd^c\Psi_j \wedge(dd^c\Phi_j)^n\wedge h
\end{align*}
and this proves that $V_j$ is plurisubharmonic on $\mathbb{B}^{nm}\times H$.

To prove $(2)$, we fix $(A,w)\in \mathbb{B}^{nm}\times H$.
First, we would like to show $V_{j_1}(A,w)\geq V_{j_2}(A,w)$, $j_1\leq j_2$.
Notice that in a neighborhood of $\{|z|\geq r_0\}$, $\chi(\Psi_{j_1})(dd^c_z\Phi_j)^n=\chi(\Psi_{j_2})(dd^c_z\Phi)^n$.
Moreover, in $\{|z|<r_0\}$, we have $\chi(\Psi_{j_1})=\Psi_{j_1}$, $\chi(\Psi_{j_2})=\Psi_{j_2}$.
To get the monotonicity, it is sufficient to prove that: $\int_{\{|z|<r_0\}}\Psi_{j_1}(dd^c_z\Phi_{j_1})^n\geq \int_{\{|z|<r_0\}}\Psi_{j_2}(dd^c_z\Phi_{j_2})^n$.
This is done by $n$ integration by parts.
More precisely, we compute that
\begin{align*}
&\int_{\{|z|<r_0\}}\Psi_{j_1}(dd^c_z\Phi_{j_1})^n\\
=&\int_{\{|z|<r_0\}}\log r_0(dd^c_z\Phi_{j_1})^n+\int_{\{|z|<r_0\}}(\Psi_{j_1}-\log r_0)(dd^c_z\Phi_{j_1})^n\\
=&\int_{\{|z|<r_0\}}\log r_0(dd^c_z\Phi_{j_2})^n+\int_{\{|z|<r_0\}}(\Psi_{j_1}-\log r_0)(dd^c_z\Phi_{j_1})^n.
\end{align*}
The second ``=" is because $\Phi_{j_1}=\Phi_{j_2}$ near $\{|z|=r_0\}$.
As $\Psi_{j_1}\geq\Psi_{j_2}$, we get
\begin{align*}
&\int_{\{|z|<r_0\}}(\Psi_{j_1}-\log r_0)(dd^c_z\Phi_{j_1})^n\\
\geq&\int_{\{|z|<r_0\}}(\Psi_{j_2}-\log r_0)(dd^c_z\Phi_{j_1})^n\\
=&\int_{\{|z|<r_0\}}(\Phi_{j_1}-M\log r_0)(dd^c_z\Psi_{j_2})\wedge(dd^c_z\Phi_{j_1})^{n-1}
\end{align*}
after an integration by parts (note that $\Phi_{j_1}-M\log r_0=\Psi_{j_2}-\log r_0=0$ on the boundary $\{|z|=r_0\}$).
Repeating this arguments, we obtain that
\begin{align*}
&\int_{\{|z|<r_0\}}(\Phi_{j_1}-M\log r_0)(dd^c_z\Psi_{j_2})\wedge(dd^c_z\Phi_{j_1})^{n-1}\\
\geq &\int_{\{|z|<r_0\}}(\Phi_{j_2}-M\log r_0)(dd^c_z\Psi_{j_2})\wedge(dd^c_z\Phi_{j_1})^{n-1}\\
=&\int_{\{|z|<r_0\}} (\Phi_{j_1}-M\log r_0)(dd^c_z\Psi_{j_2})\wedge(dd^c_z\Phi_{j_2})\wedge(dd^c_z\Phi_{j_1})^{n-2}\\
&\cdots\cdots\\
\geq&\int_{\{|z|<r_0\}} (\Phi_{j_2}-M\log r_0)(dd^c_z\Psi_{j_2})\wedge(dd^c_z\Phi_{j_2})^{n-1}\\
=&\int_{\{|z|<r_0\}} (\Psi_{j_2}-\log r_0)\wedge(dd^c_z\Phi_{j_2})^{n}.
\end{align*}
This proves that $V_{j_1}(A,w)\geq V_{j_2}(A,w)$, so $V_j=\int_{\mathbb{B}^n}\chi(\Psi_j)(dd^c_z\Phi_j)^n$ is decreasing with respect to $j$.
By Bedford-Taylor-Demailly's monotone convergence theorem, $\Psi_j(dd^c_z\Phi_j)^n\ra\Psi(dd^c_z\Phi)^n$ in the weak sense on $\mathbb{B}^n$.
As a consequence, for every $(A,w)\in\mathbb{B}^{nm}\times H$,
\begin{align*}
\int_{\mathbb{B}^n}\chi(\Psi_j)(dd^c_z\Phi_j)^n \searrow \int_{\mathbb{B}^n}\chi(\Psi)(dd^c_z\Phi)^n.
\end{align*}
The proof of claim $(2)$ is complete.
%(As a remark, although $V_j$ converge to $V$ everywhere, $V$ could differ from a plurisubharmonic function on a pluripolar set. It seems that the lack of continuity of $\varphi,\psi$ may result in this condition.)

For every $r\geq0$, the function
\begin{align*}
\mathbb{B}^{nm}\times H\ni(A,w)\mapsto V(A,w)-r\mathrm{Re}w
\end{align*}
is plurisubharmonic and independent of $\mathrm{Im}w$.
By Kiselman's minimal principle,
\begin{align*}
\mathbb{B}^{nm}\ni A\mapsto U^r(A):=\inf_{t<-2}\{V(A,t)-rt\}
\end{align*}
is plurisubharmonic on $\mathbb{B}^{nm}$.
%Similarly, we set
%\begin{align*}
%\mathbb{B}^{nm}\ni A\mapsto U^r(A):=\inf_{t<-2}\{V(A,t)-rt\}.
%\end{align*}
It directly follows that
\begin{itemize}
\item
$U^r(A)=-\infty$, if $\int_{\{0\}}(dd^c_z\Phi(z,Az))^n>r$,
\item
$U^r(A)>-\infty$, if $\int_{\{0\}}(dd^c_z\Phi(z,Az))^n<r$.
\end{itemize}
In particular, since $\int_{\{0\}}(dd^c_z\Phi(z,Az))^n=0$ holds for almost every $A\in\mathbb{B}^{nm}$, we infer that $U^r\not\equiv-\infty$.
For $r\geq 0$, we set
\begin{align*}
E_r:=\{A\in\mathbb{B}^{nm}:\int_{\{0\}}(dd^c\Phi(z,Az))^n>r\}.
\end{align*}
The above analyses yield that
\begin{itemize}
\item
$E_r\subseteq\{U^r=-\infty\}$,
\item
$\{U^r=-\infty\}\subseteq E_s$, if $s<r$.
\end{itemize}
%By definition of $U^r$, $E_r\subseteq\{U^r=-\infty\}$.
%Furthermore, since $V_j$ converge to $V$ pointwise on $\mathbb{B}^{nm}\times H$, we have
%$\limsup_{j\ra \infty}U^r_j\leq U^r$.
%Therefore, $\{U^r=-\infty\}$ is contained in a pluripolar set
%\begin{align*}
%\{\widetilde{U^r}=-\infty\}\cup \bigcup_{j=1}^{\infty}\{(\sup_{k\geq j}U^r_k)^*\neq(\sup_{k\geq j}U^r_k)\},
%\end{align*}
It follows that $E_0=\lim_{r\ra0}E_{r}=\lim_{r\ra0}\{U^r=-\infty\}$, which leads to $E_0$ is a pluripolar set,
and this completes our proof.
\end{proof}

\begin{rem}
$\int_{\{0\}}(dd^c\varphi)^n=0$ is equivalent to that,
for every $r>0$, plurisubharmonic function $U^r$ has no poles.
\end{rem}

\begin{rem}
Using the log-convexity of the sequence $\int_{\{0\}}(dd^c\varphi)^j\wedge(dd^c\max\{\varphi,A^*\psi\})^{n-j}$, $j=0,\cdots,n$,
we have the following inequality
\begin{align*}
\int_{\{0\}}(dd^c\varphi)^{n-1}\wedge (dd^c\max\{\varphi,A^*\psi\})
\leq \Big(\int_{\{0\}} dd^c\max\{\varphi,A^*\psi\}\Big)^{\frac{1}{n}}\Big(\int_{\{0\}}(dd^c\varphi)^{n}\Big)^{\frac{n-1}{n}}.
\end{align*}
Therefore we obtain the stronger conclusion:
for $A\in\mathrm{Hom}(\mc^n,\mc^m)$ outside a pluripolar set,
\begin{align*}
\int_{\{0\}}(dd^c\varphi)^{n-1}\wedge (dd^c\max\{\varphi,A^*\psi\})=0.
\end{align*}
\end{rem}

%\begin{rem}
%Note that $\int_{\{0\}}(dd^c\max\{\varphi(z),\varphi(Az)\})^n$ is the Lelong number of current $[\mc^n]$ with respect to the weight $\max\{\varphi(z),\varphi(Az)\}$.
%If the function $A\mapsto \int_{\{0\}}(dd^c\Phi(z,Az))^n$ is upper semi-continuous in the analytic Zariski topology,
%\begin{align*}
%\int_{\{0\}}(dd^c\max\{\varphi(z),\varphi(Az)\})^n=0
%\end{align*}
%will hold for generic $A\in\mc^{nm}$.
%\end{rem}

To prove Theorem \ref{thm:n-1}, we need a formula due to Z. B\l{}ocki in \cite[Theorem 4]{Blo00}.
\begin{lem} \label{lem:Blocki formula}
Suppose $\varphi,\psi\in\mbox{PSH}(\mathbb{B}^n)\cap L^{\infty}_{\mathrm{loc}}(\mathbb{B}^n)$, we have
\begin{align*}
(dd^c\max\{\varphi,\psi\})^n=dd^c\max\{\varphi,\psi\}\wedge\sum_{k=0}^{n-1}(dd^c\varphi)^k\wedge(dd^c\psi)^{n-1-k}-\sum_{k=1}^{n-1}(dd^c\varphi)^k\wedge(dd^c\psi)^{n-k}.
\end{align*}
\end{lem}
Using a standard approximation we see that the above lemma holds for $\varphi,\psi$ with isolated singularities.
\begin{proof}[Proof of Theorem \ref{thm:n-1} for $\varphi,\psi$ with isolated singularities]
We may assume that $A$ is non-degenerate,
then $A^*\psi\in\mbox{PSH}(\mathbb{B}^n)$ has isolated singularity $0$.
By Lemma \ref{lem:Blocki formula}, it holds that
\begin{align*}
&\int_{\{0\}}(dd^c\max\{\varphi,A^*\psi\})^n\\
=&\int_{\{0\}}dd^c\max\{\varphi,A^*\psi\}\wedge\sum_{k=0}^{n-1}(dd^c\varphi)^k\wedge(dd^cA^*\psi)^{n-1-k}
-\int_{\{0\}}\sum_{k=1}^{n-1}(dd^c\varphi)^k\wedge(dd^cA^*\psi)^{n-k}\\
\leq &\int_{\{0\}}dd^c\max\{\varphi,A^*\psi\}\wedge (dd^c(\varphi+A^*\psi))^{n-1}
-\sum_{k=1}^{n-1}\int_{\{0\}}(dd^c\varphi)^{k}\wedge (dd^c A^*\psi)^{n-k}.
\end{align*}
By the log-convexity of the sequence $\int_{\{0\}}(dd^c\varphi)^{j}\wedge(dd^cA^*\psi)^{n-j}$, $j=0,\cdots,n$,
we obtain that
\begin{align*}
\int_{\{0\}}dd^c\max\{\varphi,A^*\psi\}\wedge\big(dd^c(\varphi+A^*\psi)\big)^{n-1}
\leq \Big( \int_{\{0\}}(dd^c\max\{\varphi,A^*\psi\})^n \Big)^{\frac{1}{n}} \Big( \int_{\{0\}}\big(dd^c(\varphi+A^*\psi)\big)^n \Big)^{\frac{n-1}{n}}.
\end{align*}
If $\int_{\{0\}}(dd^c\max\{\varphi,A^*\psi\})^n=0$,
the above estimate implies that
\begin{align*}
\int_{\{0\}} (dd^c\varphi)\wedge(dd^cA^*\psi)^{n-1}
=\cdots=
\int_{\{0\}} (dd^c\varphi)^{n-1}\wedge (dd^cA^*\psi)=0.
\end{align*}
The proof is therefore complete.
\end{proof}

\section{A sharp estimate of the residual Monge-Amp\`ere mass}\label{sec:estimate}
%\subsection{A sharp estimate of the residual Monge-Amp\`ere mass of plurisubharmonic functions with finite log threshold}\label{subsec:sharp estimate}
In this section, we prove an upper bound for $\nu_n(\varphi)$  in terms of $\nu_1(\varphi),\cdots,\nu_{n-1}(\varphi)$ for plurisubharmonic funtions with finite log truncated threshold.
We will also show that the estimate we give is optimal in certain reasonable sense by a series of examples.

\begin{thm}(=Theorem \ref{thm:estimate})\label{thm:sharp estimate}
Suppose $\varphi \in\mbox{PSH}(\mathbb{B}^n)$ and $lt(\varphi,0)=\gamma<+\infty$.
%If there is a constant $\gamma>0$ and a sequence of positive numbers $r_k\ra0$ such that in a neighborhood $U_k$ of the sphere $\{|z|=r_k\}$,
%\begin{align*}
%\sup_{z\in U_k}\frac{\varphi(z)}{\log|z|}\leq \gamma,
%\end{align*}
Then for every $\beta=(\beta_1,\cdots,\beta_{n-1})$, $\beta_j\geq0$ and $\beta_1+\cdots+\beta_{n-1}=1$, it holds that
\begin{align*}
\nu_n(\varphi)
\leq \big(\frac{\nu_1(\varphi)}{\gamma}\big)^{\beta_1}\cdot\big(\frac{\nu_2(\varphi)}{\gamma^2}\big)^{\beta_2}\cdot...\cdot\big(\frac{\nu_{n-1}(\varphi)}{\gamma^{n-1}}\big)^{\beta_{n-1}}\cdot\gamma^{n}.
\end{align*}
In particular, if $\nu_{1}(\varphi)=0$, it holds that $\nu_n(\varphi)=0$.
\end{thm}
%We need the log-convexity of the sequence $\int_{\{0\}}(dd^c\varphi)^j\wedge (dd^c\log|z|)^{n-j}$, $j=0,\cdots,n$.
%\begin{lem}\label{lem:log convex}(\cite[Lemma 2.1]{DH14})
%Suppose $\varphi \in\mbox{PSH}(\mathbb{B}^n)\cap L^{\infty}_{\mathrm{loc}}(\mathbb{B}^n\backslash\{0\})$ and set $e_j(\varphi):=\int_{\{0\}}(dd^c\varphi)^j\wedge(dd^c\log|z|)^{n-j}$, then for $j=1,\cdots,n-1$,
%\begin{align*}
%e_j(\varphi)^2\leq e_{j-1}(\varphi)e_{j+1}(\varphi).
%\end{align*}
%\end{lem}
\begin{proof}
Fix $\gamma_0>\gamma'_0>\gamma$ and $m\geq 1$.
We consider a plurisubharmonic function $\Phi$ on $\mathbb{B}^n\times\mathbb{B}^m$ as
$$\Phi(z,w):=\max\{\varphi(z),\gamma_0\log|w|\}.$$
As in the proof of Theorem \ref{thm:max}, for $1\leq j\leq m+n-1$ we have
\begin{align*}
e_{j}(\Phi)
:&=\int_{\{0\}}(dd^c\Phi)^{j}\wedge (dd^c\log|(z,w)|)^{m+n-j}\\
&=\int_{\{0\}}(dd^c\Phi)^{j}\wedge (dd^c\log|(z,w)|)^{n-j}\wedge (dd^c\log|(z,w)|)^{m}\\
&=\int_{\{0\}}(dd^c\Phi)^{j}\wedge (dd^c\log|(z,w)|)^{n-j}\wedge\int_{G(n,n+m)}[S]dv(S)\\
&=\int_{S\in G(n,n+m)}\Big(\int_{\{0\}}(dd^c\Phi)^j\wedge(dd^c\log|(z,w)|)^{n-j}\wedge[S]\Big)dv(S).
\end{align*}

\emph{Claim.} For $j=1,\cdots,n-1$, $\nu_{j}(\varphi)\geq e_j(\Phi)$.

The proof of the claim is as follows.
We use $\mbox{Hom}(\mc^n,\mc^m)=\mc^{nm}$ to paramatrize a Zariski neighborhood of $\mc^n\times\{0\}$ in $G(n,n+m)$ as before.
If $S$ is the graph of the mapping $A:\mc^n\ra\mc^m,z\mapsto Az$, we have the following estimate
\begin{align*}
&\int_{\{0\}}(dd^c\Phi)^j\wedge(dd^c\log|(z,w)|)^{n-j}\wedge[S]\\
=&\int_{\{0\}}(dd^c\max\{\varphi(z),\log|Az|\})^j\wedge(dd^c\log|(z,Az)|)^{n-j}\\
\leq&\int_{\{0\}}(dd^c\varphi)^j\wedge (dd^c\log|z|)^{n-j},
\end{align*}
where the inequality is due to the comparison theorem for Lelong numbers.
Since $dv $ has total mass $1$, the claim is therefore proved.

Lemma \ref{lem:mix} implies that
\begin{align*}
\frac{e_2(\Phi)}{e_1(\Phi)}\leq \cdots
\leq \frac{e_{n}(\Phi)}{e_{n-1}(\Phi)}
\leq \frac{e_{n+1}(\Phi)}{e_n(\Phi)}\leq \cdots
\leq \frac{e_{n+m}(\Phi)}{e_{n+m-1}(\Phi)},
\end{align*}
which leads to
\begin{align*}
&\Big(\frac{e_{n}(\Phi)}{e_1(\Phi)}\Big)^m
=\Big( \frac{e_{n}(\Phi)}{e_{n-1}(\Phi)}\cdot...\cdot \frac{e_2(\Phi)}{e_1(\Phi)}\Big)^m
\leq \Big( \frac{e_{n+m}(\Phi)}{e_{n+m-1}(\Phi)}\cdot...\cdot\frac{e_{n+1}(\Phi)}{e_n(\Phi)}   \Big)^{n-1}
=\Big(\frac{e_{n+m}(\Phi)}{e_n(\Phi)}\Big)^{n-1},\\
&\Big(\frac{e_{n}(\Phi)}{e_2(\Phi)}\Big)^m
=\Big( \frac{e_{n}(\Phi)}{e_{n-1}(\Phi)}\cdot...\cdot \frac{e_3(\Phi)}{e_2(\Phi)}\Big)^m
\leq \Big( \frac{e_{n+m}(\Phi)}{e_{n+m-1}(\Phi)}\cdot...\cdot\frac{e_{n+1}(\Phi)}{e_n(\Phi)}   \Big)^{n-2}
=\Big(\frac{e_{n+m}(\Phi)}{e_n(\Phi)}\Big)^{n-2},\\
&\cdots\cdots\\
&\Big(\frac{e_{n}(\Phi)}{e_{n-1}(\Phi)}\Big)^m
\leq \Big(\frac{e_{n+m}(\Phi)}{e_{n+m-1}(\Phi)}\cdot...\cdot\frac{e_{n+1}(\Phi)}{e_n(\Phi)}\Big)=\Big(\frac{e_{n+m}(\Phi)}{e_n(\Phi)}\Big).
\end{align*}
Consequently, for every non-negative number $\alpha_j$, $j=1,\cdots,n-1$,
\begin{align*}
\Big(\frac{e_n(\Phi)}{e_j(\Phi)}\Big)^{m\alpha_j}\leq\Big(\frac{e_{n+m}(\Phi)}{e_n(\Phi)}\Big)^{(n-j)\alpha_j}.
\end{align*}
We multiply the above inequalities and get that
\begin{align*}
e_n(\Phi)
\leq \Big(\Pi_{j=1}^{n-1}e_j(\Phi)^{\frac{m\alpha_j}{m(\alpha_1+\cdots+\alpha_{n-1})+(n-1)\alpha_1+\cdots+\alpha_{n-1}}} \Big)
\cdot e_{n+m}(\Phi)^{\frac{(n-1)\alpha_1+\cdots+\alpha_{n-1}}{m(\alpha_1+\cdots+\alpha_{n-1})+(n-1)\alpha_1+\cdots+\alpha_{n-1}}}.
\end{align*}
%On one hand we have $e_1(\Phi)=\nu_{\Phi}(0)\leq\nu_1(\varphi)$.
We claim that $e_{n+m}(\Phi)\leq\gamma_0^{n+m}$.
Our assumption $lt(\varphi,0)=\gamma$ means there is a sequence of positive numbers $r_k\ra0$ so that $\varphi(z)\geq \gamma'_0\log|z|$ in a neighborhood $U_k$ of the sphere $\{|z|=r_k\}$ in $\mc^n$.
For $(z,w)\in U_k\times\{|w|<r_k\}$,
\begin{align*}
\Phi(z,w)\geq\max\{\gamma_0\log|z|,\gamma_0\log|w|\}\geq \gamma_0\log|(z,w)|-\gamma_0\log4,
\end{align*}
and for $(z,w)\in \{|z|<r_k\}\times\{\frac{r_k}{2}<|w|<r_k\}$,
\begin{align*}
\Phi(z,w)\geq\gamma_0\log|w|\geq \gamma_0\log r_k-\gamma_0\log2\geq \gamma_0\log|(z,w)|-\gamma_0\log4.
\end{align*}
%Since $\Phi$ and $\max\{\Phi,\gamma_0\log|(z,w)|-\gamma_0\log4\}$ coincide on a neighborhood of $\{|z|<r_k\}\times\{|w|<r_k\}$,
An integration by parts yields that
\begin{align*}
\int_{\{|z|<r_k\}\times\{|w|<r_k\}}(dd^c\Phi)^{n+m}=\int_{\{|z|<r_k\}\times \{|w|<r_k\} }\big(dd^c\max\{\Phi, \gamma_0\log|(z,w)|-\gamma_0\log4\}\big)^{n+m}.
\end{align*}
Letting $r_k\ra0^+$, we obtain that
\begin{align*}
\int_{\{0\}}(dd^c\Phi)^{n+m}
&=\int_{\{0\}}(dd^c\max\{\Phi,\gamma_0\log|(z,w)|- \gamma_0\log4\})^{n+m}\\
&\leq \int_{\{0\}}(dd^c\gamma_0\log|(z,w)|)^{n+m}\\
&= \gamma_0^{n+m},
\end{align*}
where the comparison theorem gives the second inequality.
The claim is thus proved.

As a consequence, we obtain the following estimate
\begin{align}\label{eq:en estimate}
e_{n}(\Phi)&
%=\int_{S\in G(n,n+m)}\Big(\int_{\{0\}}(dd^c\Phi)^n\wedge[S]\Big)dv(S)\\&
\leq
\Big(\Pi_{j=1}^{n-1}e_j(\Phi)^{\frac{m\alpha_j}{m(\alpha_1+\cdots+\alpha_{n-1})+(n-1)\alpha_1+\cdots+\alpha_{n-1}}} \Big)
\cdot \gamma_0^{(n+m)(\frac{(n-1)\alpha_1+\cdots+\alpha_{n-1}}{m(\alpha_1+\cdots+\alpha_{n-1})+(n-1)\alpha_1+\cdots+\alpha_{n-1}})}.
\end{align}
Since $dv$ has total mass $1$, there is some $A\in \mc^{nm}$ such that
\begin{align*}
\int_{\{0\}}(dd^c\max\{\varphi(z),\gamma_0\log|Az|\})^n
\leq e_{n}(\Phi)  .
\end{align*}
If $z\in U_k$ and $r_k$ is sufficiently small,
we have $\varphi(z)\geq \gamma'_0\log|z|\geq \gamma_0\log|Az|$.
Hence in $U_k$, $\varphi(z)$ coincides with $\max\{\varphi(z),\gamma_0\log|Az|\}$.
It follows that
\begin{align*}
\int_{\{0\}}(dd^c\varphi)^n
&=\lim_{r_k\ra0}\int_{B(0,r_k)}(dd^c\varphi)^n\\
&=\lim_{r_k\ra0}\int_{B(0,r_k)}(dd^c\max\{\varphi,\gamma_0\log|Az|\})^n\\
&=\int_{\{0\}}(dd^c\max\{\varphi,\gamma_0\log|Az|\})^n\\
&\leq e_n(\Phi),
\end{align*}
where the second equality is given by the integration-by-parts argument.
Letting $\gamma_0\ra \gamma$ and $m\ra\infty$, we obtain that
\begin{align*}
\int_{\{0\}}(dd^c\varphi)^n
&\leq e_1(\Phi)^{\frac{\alpha_1}{\alpha_1+\cdots+\alpha_{n-1}}}\cdot...\cdot e_{n-1}(\Phi)^{\frac{\alpha_{n-1}}{\alpha_1+\cdots+\alpha_{n-1}}}\cdot\gamma^{\frac{(n-1)\alpha_1+\cdots+\alpha_{n-1}}{\alpha_1+\cdots+\alpha_{n-1}}}\\
&\leq \nu_1(\varphi)^{\frac{\alpha_1}{\alpha_1+\cdots+\alpha_{n-1}}}\cdot...\cdot \nu_{n-1}(\varphi)^{\frac{\alpha_{n-1}}{\alpha_1+\cdots+\alpha_{n-1}}}\cdot\gamma^{\frac{(n-1)\alpha_1+\cdots+\alpha_{n-1}}{\alpha_1+\cdots+\alpha_{n-1}}}\\
&=(\frac{\nu_1(\varphi)}{\gamma}\big)^{\frac{\alpha_1}{\alpha_1+\cdots+\alpha_{n-1}}}\big(\frac{\nu_2(\varphi)}{\gamma^2}\big)^{\frac{\alpha_2}{\alpha_1+\cdots+\alpha_{n-1}}}\cdot...\cdot \big(\frac{\nu_{n-1}(\varphi)}{\gamma^{n-1}}\big)^{\frac{\alpha_{n-1}}{\alpha_1+\cdots+\alpha_{n-1}}}\cdot\gamma^{n}.
\end{align*}
where the fist inequality follows from \eqref{eq:en estimate}, and the second one follows from Claim.
Seting $\beta=(\beta_1,\cdots,\beta_{n-1}):=(\frac{\alpha_1}{\alpha_1+\cdots+\alpha_{n-1}},\cdots,\frac{\alpha_{n-1}}{\alpha_1+\cdots+\alpha_{n-1}})$, the proof is complete.
\end{proof}
We study a few examples of plurisubharmonic functions with isolated singularities, which reveals that the estimate in Theorem \ref{thm:estimate} is sharp.
\begin{example}
Let $\varphi(z)=\max\{\log|z_1|^{\alpha_1},\cdots,\log|z_n|^{\alpha_n}\}$ with $0\leq\alpha_1\leq\alpha_2\leq\cdots\leq\alpha_{n-1}$, then $\nu_1(\varphi)=\alpha_1$, $\nu_2(\varphi)=\alpha_1\alpha_2,\cdots,\nu_{n-1}(\varphi)=\alpha_1\cdot...\cdot\alpha_{n-1}$ and $\gamma=\alpha_n$.
Theorem \ref{thm:estimate} gives us the following nontrivial combinatorial inequality
\begin{align*}
\alpha_1\cdot...\cdot\alpha_n
\leq
\left(\frac{\alpha_1}{\alpha_n}\right)^{\beta_1}
\left(\frac{\alpha_1\alpha_2}{\alpha_n^2}\right)^{\beta_2}
\cdot...\cdot
\left(\frac{\alpha_1\cdot...\cdot\alpha_{n-1}}{\alpha_n^{n-1}}\right)^{\beta_{n-1}}
\alpha_n^n,
\end{align*}
which is optimal in certain cases depending on the choices of $\beta_1,\cdots, \beta_{n-1}$.
\end{example}
\begin{example}
Let $\varphi(z_1,z_2)=\frac{1}{2m}\log(|z_1-z_2^m|^2+|z_1^m|^2)$, $m\in\mathbb{N}^*$, then $\nu_1(\varphi)=\frac{1}{m}$ and $\gamma=m
$.
Theorem \ref{thm:estimate} reads as
\begin{align*}
1\leq \frac{1}{m} \cdot m=1.
\end{align*}
This example shows that we cannot expect $ \int_{\{0\}} (dd^c\varphi)^2 \leq \nu_1(\varphi)^{\alpha}\gamma^{2-\alpha}$ for $\alpha>1$.
\end{example}

\begin{example}
The function $\varphi$ considered here is constructed in \cite{LC24}.
Let $\pi:\mc^2\backslash\{0\}\ra\mathbb{P}^1$ be the canonical projection, $\omega_0$ be the Fubini-Study metric on $\mathbb{P}^1$.
Consider $u\in\mbox{PSH}(\mathbb{P}^1,\omega_0)$ so that $u$ is $\omega_0$-pluriharmonic outside $C=\{u=-\infty\}$: a Cantor set, then define
\begin{align*}
\varphi:=\max\{\log|z|^2+\pi^*u,2\log|z|^2\}.
\end{align*}
It is shown in \cite{LC24} that $\nu_2(\varphi)=8$.
It is easy to see that $\nu_1(\varphi)=2$ and $\varphi\geq\gamma\log|z|$ for $\gamma=4$, hence Theorem \ref{thm:estimate} reads as $\nu_2(\varphi)=8\leq 2\cdot4=\nu_1(\varphi)\cdot\gamma$.
This reveals the estimate is sharp again.

\end{example}

\begin{rem}
There is $\varphi\in\mbox{PSH}(\mathbb{B}^n)\cap L^{\infty}_{\mathrm{loc}}(\mathbb{B}^n\backslash\{0\})$ satifying the requirements in Theorem \ref{thm:sharp estimate}, however, $\varphi$ is not bounded from below by any $\gamma\log|z|-C$ $(\gamma,C>0)$.
Indeed, we can select a sequence of $\{z_k\}$ such that $r_k<|z_k|<r_{k+1}$ and consider $\varphi(z):=\sum_{k=1}^{\infty}\varepsilon_k\max{\{\log|z-z_k|,-N_k\}}$, where $\varepsilon_k>0$ are sufficiently small, $N_k>0$ are sufficiently large.
\end{rem}

\section{Relations to known results}\label{sec:zero mass conjecture}
In this section,
we apply Theorem \ref{thm:max final version} to recover or improve some known results in the literatures related to the zero mass conjecture of Guedj-Rashkovskii.

For $\varphi\in \mbox{PSH}(\mathbb{B}^n)$,
we say $\varphi$ is $S^1$-invariant if $\varphi(e^{i\theta}z)=\varphi(z)$, $\forall \theta\in\mr$.
\begin{lem}\label{lem:radial sh}
Suppose $u\in \mbox{SH}(\mathbb{B})$ and $u(w)=u(|w|)$, then the function $t\mapsto u(e^t)$, $t\in(-\infty,0)$ is convex and increasing.
In particular, the right-derivative $t\mapsto \partial^+_t\varphi(e^t)$ exists, is non-negative and non-decreasing with respect to $t$.
\end{lem}
\begin{proof}
For $u\in\mbox{SH}(\mathbb{B})$,
it is well-known that the function $t\mapsto \sup_{|z|=e^t}u(e^t)$, $t\in(-\infty,0)$ is convex and increasing.
The conclusion follows immediately since $u$ is radial.
\end{proof}
\begin{defn}(=Definition \cite[Definition 2.6, Definition 2.7]{HLX23})
The maximal directional Lelong number of $\varphi$ at a distance of $t\in(-\infty,0)$ is defined as
\begin{align*}
M_{t}(\varphi):=\sup_{|z|=1}\partial^+_t\varphi(e^{t}z).
\end{align*}
Since $M_t(\varphi)$ is non-decreasing with respect to $t$, we can take the limit as $t\ra-\infty$.
The maximal directional Lelong number of $\varphi$ is defined as
\begin{align*}
\lambda_{\varphi}(0):=\lim_{t\ra-\infty}M_{t}(\varphi).
\end{align*}
\end{defn}

\begin{cor}(=Corollary \ref{cor:S1 invariant})
Suppose $\varphi\in\mbox{PSH}(\mathbb{B}^n)\cap L^{\infty}_{\mathrm{loc}}(\mathbb{B}^n\backslash\{0\})$,
if $\varphi$ is $S^1$-invariant, it holds that
\begin{align*}
\int_{\{0\}}(dd^c\varphi)^n\leq\nu_1(\varphi)\lambda_{\varphi}(0)^{n-1}.
\end{align*}
In particular, if $\nu_{1}(\varphi)=0$, then $\nu_n(\varphi)=0$.
\end{cor}
\begin{proof}
The $S^1$-invariance of $\varphi$ implies that for every $z\neq0$, the function
$$\mathbb{B}\ni w\mapsto \varphi(w\cdot \frac{z}{|z|})$$
is radial.
By Lemma \ref{lem:radial sh}, for $|z|<e^t$ and $t<-2$, it holds that
\begin{align*}
\frac{\varphi(z)-\varphi(e^{t}\cdot\frac{z}{|z|})}{\log|z|-t}
\leq \partial^+_t\varphi(e^t\frac{z}{|z|})
\leq \frac{\varphi(e^{-2}\cdot\frac{z}{|z|})-\varphi(e^{-1}\cdot\frac{z}{|z|})}{(-2)-(-1)}.
\end{align*}
The second inequality above implies that $M_t(\varphi) \leq \sup_{|z|=\frac{1}{e}}\big(\varphi(z)-\varphi(\frac{z}{e})\big)<+\infty$ for $t<-2$.
Fixing arbitrary $\lambda_0>\lambda_{\varphi}(0)$, there exists $t_0$ sufficiently small such that $M_{t_0}(\varphi)\leq \lambda_0$.
Using the first inequality above, we can deduce that for $|z|<e^{t_0}$
\begin{align*}
\varphi(z) &\geq M_{t_0}(\varphi)\log|z|+\inf_{|z|=e^{t_0}}\varphi(z)\\
&\geq \lambda_0 \log|z|+\inf_{|z|=e^{t_0}}\varphi(z).
\end{align*}
Theorem \ref{thm:sharp estimate} then yields that
\begin{align*}
\int_{\{0\}}(dd^c\varphi)^n \leq \nu_1(\varphi)\lambda_0^{n-1}.
\end{align*}
Letting $\lambda_0\searrow \lambda_{\varphi}(0)$, the proof is complete.
\end{proof}

In classical potential theory, it is well-known that any $u\in\mbox{SH}(\Omega)$ belongs to $W^{1,1+\varepsilon}_{\mathrm{loc}}(\Omega)$ for some $\varepsilon>0$,
therefore the first order derivatives of $u$ is well-defined.
\begin{defn}\label{def:uniform Lip}(=\cite[Definition 2.7]{HLX24})
A function $\varphi\in\mbox{PSH}(\mathbb{B}^n)$ is said to be uniformly directional Lipschitz continuous near the origin if the derivative $r\varphi_r$ is an $L^{\infty}$ function in a smaller ball, i.e. we have
\begin{align*}
L_t(\varphi):=\|r\varphi_r\|_{L^{\infty}(\overline{B_R})}<+\infty
\end{align*}
for some $t=\log R<0$.
Since $L_t(\varphi)$ is non-decreasing in $t$, we define the directional Lipschitz constant of $\varphi$ at the origin as
\begin{align*}
\kappa_{\varphi}(0):=\lim_{t\ra-\infty}L_t(\varphi).
\end{align*}
%If $\kappa_{\varphi}(0)<+\infty$, we say that $\varphi$ is uniformly directional Lipschitz continuous near the origin.
\end{defn}
We improve \cite[Theorem 1.1]{HLX24} as follows.
\begin{cor}(=Corollary \ref{cor:uniform Lip})
Suppose $\varphi\in\mbox{PSH}(\mathbb{B}^n)$ is uniformly directional Lipschitz continuous near the origin,
then for any $\gamma>\kappa_{\varphi}(0)$ we have $\varphi\geq\gamma\log|z|+O(1)$ near the origin.
Moreover the following estimate for residue mass holds:
\begin{align*}
\int_{\{0\}}(dd^c\varphi)^n\leq \nu_1(\varphi)\kappa_{\varphi}(0)^{n-1}.
\end{align*}
\end{cor}
We need the following lemma to give an upper bound for $r\partial_r\varphi_{\varepsilon}$, where $\varphi_{\varepsilon}$ is the standard smooth approximation of $\varphi$.
\begin{lem}
Suppose $\varphi\in\mbox{PSH}(\mathbb{B}^n)$ is uniformly directional Lipschitz continuous near the origin, then for $t<-1$ and $\varepsilon $ small,
$L_t(\varphi_{\varepsilon})\leq L_{t+1}(\varphi)+n\varepsilon \|D\varphi\|_{L^1(B_{R})}$ with $\log R=t+1$.
\end{lem}
\begin{proof}
By definition, $r\partial_r\varphi=\sum_{k=1}^{2n}x^k\partial_{x^k}\varphi$.
For $\varphi_{\varepsilon}=\varphi*\rho_{\varepsilon}$, $r\partial_r \varphi_{\varepsilon}
=r(\partial_r\varphi*\rho_{\varepsilon})
=\sum_{k=1}^{2n}x^k(\partial_{x^k}\varphi*\rho_{\varepsilon})$,
then for $|x|\leq e^t$ and $\varepsilon<e^{t+1}-e^{t}$, we compute that
\begin{align*}
\Big|x^k(\partial_{x^k}\varphi*\rho_{\varepsilon})-(x^k\partial_{x^k}\varphi)*\rho_{\varepsilon}\Big|(x)
&\leq\int_{|y|<1}|x^k-(x^k-\varepsilon y^k)|\cdot|\partial_{x^k}\varphi(x-\varepsilon y)|\rho(y)dV(y)\\
&\leq\varepsilon\|D\varphi\|_{L^1(B_{e^{t+1}})}.
\end{align*}
The estimate $L_t(\varphi_{\varepsilon})\leq L_{t+1}(\varphi)+n\varepsilon\|D\varphi\|_{L^1(B_R)}$ follows as desired.
\end{proof}
\begin{proof}
Since $\gamma>\kappa_{\varphi}(0)$, there is a negative number $t_0<0$ so that $\gamma> L_{t_0+1}(\varphi)$.
Fix $|\xi|=1$, we consider the function defined on $(-\infty,0):t\mapsto\varphi(e^t\xi)$.
For any $t<t_0$, we have $\varphi(e^t\xi)-\varphi(e^{t_0}\xi)=\lim_{\varepsilon\ra0}\varphi_{\varepsilon}(e^t\xi)-\varphi_{\varepsilon}(e^{t_0}\xi)=\lim_{\varepsilon\ra0}-\int_{t}^{t_0}\partial_t (\varphi_{\varepsilon}(e^t\xi))dt\geq L_{t_0+1}(\varphi)(t-t_0)$.
Since $\xi$ is arbitrary, we see that for every $z\in\mathrm{B}^n$, $|z|<e^{t_0}$,
\begin{align*}
\varphi(z)
&\geq L_{t_0+1}(\varphi)\log|z|-t_0L_{t_0+1}(\varphi)+\inf_{|w|=e^{t_0}}\varphi(w)\\
&\geq \gamma\log|z|+O(1).
\end{align*}
The estimate for the residue mass follows from Theorem \ref{thm:estimate} immediately.
\end{proof}

If $e^{\varphi}$ is H\"older continuous near the origin, \cite{BFJ08}, \cite{Ra13-1} have proved that $\nu_1(\varphi)=0$ implies $\nu_n(\varphi)=0$.
\begin{cor}
Suppose $\varphi\in\mbox{PSH}(\mathbb{B}^n)$, $e^{\varphi}$ is H\"older continuous.
If $\nu_1(\varphi)=0$, then $\varphi\geq\gamma\log|z|+O(1)$ and $\nu_n(\varphi)=0$.
\end{cor}
\begin{proof}
Suppose $e^{\varphi}$ is $\alpha$-H\"older continuous.
The assumption $\nu_1(\varphi)=0$ implies that on the sphere $\{|z|=r\}$ with $r$ small enough, there is some $z_0$ such that $\varphi(z_0)\geq\frac{\alpha}{2}\log|z_0|=\frac{\alpha}{2}\log r$.
Note that $|z-z_0|\leq 2r$ if $|z|=r$, hence $|e^{\varphi(z)}-e^{\varphi(z_0)}|\leq C(2r)^{\alpha}$ gives $\varphi(z)\geq \alpha\log|z|+O(1)$.
Therefore Theorem \ref{thm:estimate} leads us to $\nu_n(\varphi)=0$ as desired.
\end{proof}

\section*{Appendix: the proofs of Theorem \ref{thm:max final version} and Theorem \ref{thm:n-1} in the full generality}
We prove Theorem \ref{thm:max final version} and Theorem \ref{thm:n-1} for  their full generality in this appendix.

The seminal papers \cite{Ce98},\cite{Ce04} introduced a class of plurisubharmonic functions on bounded hyperconvex domains, now known as the Cegrell class, which is defined as follows.
\begin{defn}\label{def:Cegrell class}
Let $\Omega\subset\mc^n$ be a bounded hyperconvex domain, we define
\begin{align*}
\mathcal{E}_0(\Omega)
&=\{\varphi\in\mbox{PSH}^-(\Omega)\cap L^{\infty}(\Omega):\lim_{z\ra\partial\Omega}\varphi(z)=0,\int_{\Omega}(dd^c\varphi)^n<+\infty\},\\
\mathcal{F}(\Omega)
&=\{\varphi\in\mbox{PSH}^-(\Omega):\exists \ \mathcal{E}_0(\Omega)\ni \varphi_j\searrow\varphi,\ \sup_j\int_{\Omega}(dd^c\varphi_j)^n<+\infty\},\\
\mathcal{E}(\Omega)
&=\{\varphi\in\mbox{PSH}^-(\Omega):\exists \ \varphi_K\in\mathcal{F}(\Omega)\ \mathrm{such}\ \mathrm{that}\ \varphi_K=\varphi \ \mathrm{on}\ K,\forall\ K\subseteq\subseteq \Omega\}.
\end{align*}
We call $\mathcal{E}(\Omega)$ the Cegrell class of $\Omega$.
\end{defn}
The importance of the Cegrell class lies in the fact that:
if $\varphi_j\in\mathcal E(\Omega),\ 1\leq j\leq n$, then the mixed Monge-Amp\`ere measure $dd^c\varphi_1\wedge\cdots\wedge dd^c\varphi_n$ is well defined as a positive measure on $\Omega$
and satisfies a series of good properties.

We need the following fact that is proved in \cite[Corollary 2.1]{ACP10}, \cite[Theorem 5.7]{Wi05},
for the sake of the completeness, we provide a proof here.
\begin{lem}
Suppose $\varphi\in\mathcal{E}(\mathbb{B}^n),\psi\in \mathcal{E}(\mathbb{B}^m)$,
viewing $\varphi,\psi$ as functions on $\mathbb{B}^n\times\mathbb{B}^m$,
then $\max\{\varphi,\psi\}\in \mathcal{E}(\mathbb{B}^n\times\mathbb{B}^m)$.
\end{lem}
\begin{proof}
Without loss of generality, we may assume $\varphi,\psi$ are in the class of $\mathcal{F}$.
By definition, there exists $\varphi_j\in\mathcal{E}_0(\mathbb{B}^n),\psi_j\in \mathcal{E}_0(\mathbb{B}^m)$ decreasing to $\varphi,\psi$ such that $\sup_j\int_{\mathbb{B}^n}(dd^c\varphi_j)^n<+\infty$, $\sup_j\int_{\mathbb{B}^m}(dd^c\psi_j)^m<+\infty$.
We consider
\begin{align*}
\tilde{\varphi}_j&:=\big(\sup\{u\in\mbox{PSH}^-(\mathbb{B}^n):u\leq\varphi_j\ \mathrm{on}\ (1-\frac{1}{j})\mathbb{B}^n\}\big)^*,\\
\tilde{\psi}_j&:=\big(\sup\{v\in\mbox{PSH}^-(\mathbb{B}^m):u\leq\psi_j\ \mathrm{on}\ (1-\frac{1}{j})\mathbb{B}^m\}\big)^*.
\end{align*}
Clearly, $\tilde{\varphi}_j$ still decreases to $\varphi$.
A standard balayage method yields that $(dd^c\tilde{\varphi}_j)^n=0$ on $\{1-\frac{1}{j}<|z|<1\}$.
Since $\tilde{\varphi}_j\geq\varphi_j$, comparison principle yields that $\int_{\mathbb{B}^n}(dd^c\varphi_j)^n\geq\int_{\mathbb{B}^n}(dd^c\tilde{\varphi}_j)^n$.
The same conditions hold for $\tilde{\psi_j}$.
In the rest of the proof, we replace $\varphi_j,\psi_j$ by $\tilde{\varphi}_j,\tilde{\psi}_j$.
We would like show that for $j=1,2,\cdots$,
\begin{align*}
\int_{\mathbb{B}^n\times\mathbb{B}^m}(dd^c\max\{\varphi_j,\psi_j\})^{n+m}
=\big(\int_{\mathbb{B}^n}(dd^c\varphi_j)^n\big)\big(\int_{\mathbb{B}^m}(dd^c\psi_j)^m\big).
\end{align*}
We choose $\delta=\delta(j)>0$ sufficiently small, so that $\{\varphi_j>-\delta\}\subseteq\subseteq\{1-\frac{1}{j}<|z|<1\}$ and $\{\psi_j>-\delta\}\subseteq\subseteq\{1-\frac{1}{j}<|w|<1\}$.
Since Monge-Amp\`ere products are local in the plurifine topology, it holds that $\mathds{1}_{\{\varphi_j>-\delta\}}(dd^c\varphi_j)^n=\mathds{1}_{\{\varphi_j>-\delta\}}(dd^c\max\{\varphi_j,-\delta\})^n$.
This implies $(dd^c\max\{\varphi_j,-\delta\})^n=0$ on $\{\varphi_j>-\delta\}$.
Similarly, $(dd^c\max\{\psi_j,-\delta\})^m=0$ on $\{\psi_j>-\delta\}$.
By \cite[Theorem 7]{Blo00}, $(dd^c\max\{\varphi_j,\psi_j,-\delta\})^{n+m}=(dd^c\max\{\varphi_j,-\delta\})^n\wedge(dd^c\max\{\psi_j,-\delta\})^m$ on $\mathbb{B}^n\times\mathbb{B}^m$.
In particular,
\begin{align*}
\int_{\mathbb{B}^n\times\mathbb{B}^m}(dd^c\max\{\varphi_j,\psi_j,-\delta\})^{n+m}
=\big(\int_{\mathbb{B}^n}(dd^c\max\{\varphi_j,-\delta\})^n\big)\big(\int_{\mathbb{B}^m}(dd^c\max\{\psi_j,-\delta\})^m\big).
\end{align*}
Note that $\max\{\varphi_j,\psi_j,-\delta\}=\max\{\varphi_j,\psi_j\}$ near $\partial(\mathbb{B}^n\times\mathbb{B}^m)$, $\max\{\varphi_j,-\delta\}=\varphi_j$ near $\partial\mathbb{B}^n$ and $\max\{\psi_j,-\delta\}=\psi_j$ near $\partial\mathbb{B}^m$, an integration by parts leads to the desired equation.
As a consequence, $\sup_j\int_{\mathbb{B}^n\times\mathbb{B}^m}(dd^c\max\{\varphi_j,\psi_j\})^{n+m}<+\infty$, the proof is complete.
\end{proof}

%\begin{thm}\label{thm:max for E}(=Theorem \ref{thm:max final version})
%Suppose $\varphi\in\mathcal{E}(\mathbb{B}^n),\psi\in \mathcal{E}(\mathbb{B}^m)$.
%If $\nu_1(\varphi)=0$, then for $A\in\mathrm{Hom}(\mc^n,\mc^m)\cong \mathbb{C}^{nm}$ outside a pluripolar set, it holds that
%\begin{align*}
%\nu_n(\max\{\varphi,A^*\psi\}\big) =0.
%\end{align*}
%\end{thm}

%\begin{thm}\label{thm:n-1 for E}(=Theorem \ref{thm:n-1})
%Suppose $\varphi,\psi \in \mathcal{E}(\mathbb{B}^n)$.
%If $\nu_1(\varphi)=0$, then for $A\in \mathrm{Hom}(\mc^n,\mc^n)\cong \mc^{n^2}$ outside a pluripolar set, it holds that
%\begin{align*}
%\int_{\{0\}} (dd^c\varphi)\wedge(dd^cA^*\psi)^{n-1}
%=\cdots=
%\int_{\{0\}} (dd^c\varphi)^{n-1}\wedge (dd^cA^*\psi)=0.
%\end{align*}
%In particular, for $A\in \mathrm{Hom}(\mc^n,\mc^n)\cong \mc^{n^2}$ outside a pluripolar set, it holds that
%\begin{align*}
%\int_{\{0\}} (dd^c\varphi)\wedge(dd^cA^*\varphi)^{n-1}
%=\cdots=
%\int_{\{0\}} (dd^c\varphi)^{n-1}\wedge (dd^cA^*\varphi)=0.
%\end{align*}
%\end{thm}

We now give the proof of Theorem \ref{thm:max final version}.
\begin{proof}
The proof is divided into several steps.
\begin{itemize}
\item
Step 1. Theorem \ref{thm:max final version} holds for almost every $A\in\mc^{nm}$.
\end{itemize}

Since $\max\{\varphi,\psi\}\in\mathcal{E}(\mathbb{B}^n\times\mathbb{B}^m)$, we can proceed as in the proof of Theorem \ref{thm:max} to derive
\begin{align*}
0=\int_{\{0\}}(dd^c\max\{\varphi,\psi\})^n\wedge(dd^c\log|(z,w)|)^m.
\end{align*}
Let $R_0>0$ be fixed.
There is $r_0>0$ so that $|Az|<1$ provided that $\|A\|<R_0,|z|<r_0$, which leads to $\max\{\varphi,A^*\psi\}$ is well-defined on $r_0\mathbb{B}^n$ for $\|A\|<R_0$.
Since belonging to $\mathcal{E}$ is a local property by \cite{Blo06}, we have $\varphi\in\mathcal{E}(r_0\mathbb{B}^n)$.
Further, $\int_{\{0\}}(dd^c\max\{\varphi,A^*\psi\})^n$ only depends on the value of $\varphi$ near the origin, we may even assume $\varphi\in\mathcal{F}(r_0\mathbb{B}^n)$,
thus there is a sequence $\varphi_j\in\mathcal{E}_0(r_0\mathbb{B}^n)$ decreasing to $\varphi$ and $\sup_j\int_{r_0\mathbb{B}^n}(dd^c\varphi_j)^n<+\infty$.
%Note that $\max\{\varphi_j,A^*\psi\}$ decreases to $\max\{\varphi,A^*\psi\}$.
Comparison principle implies $$\int_{r_0\mathbb{B}^n}(dd^c\max\{\varphi_j,A^*\psi\})^n\leq \int_{r_0\mathbb{B}^n}(dd^c\varphi_j)^n,$$
therefore the set
$$\left\{\int_{r_0\mathbb{B}^n}(dd^c\max\{\varphi_j,A^*\psi\})^n:j\in\mathbb{N}^*,\|A\|<R_0\right\}$$
is bounded.
The graph of the mappings $A:\mc^n\ra\mc^m,z\mapsto Az$ for all $\|A\|<R_0$, given a Zariski open set in $G(n,n+m)$,  denoted by $U_0$.
Let $\chi_r$ be a sequence of cut-off functions decreasing to $\mathds{1}_{\{0\}}$ on $\mc^{nm}$, we then compute that
\begin{align*}
0&=\int_{\{0\}}(dd^c\max\{\varphi,\psi\})^n\wedge(dd^c\log|(z,w)|)^m\\
&=\lim_{r\ra0}\lim_{j\ra\infty}\int_{\mc^{nm}}\chi_r(dd^c\max\{\varphi_j,\psi\})^n\wedge\int_{G(n,n+m)}[S]dv(S)\\
&=\lim_{r\ra0}\lim_{j\ra\infty}\int_{G(n,n+m)} \big(\int_{\mc^{nm}}\chi_r(dd^c\max\{\varphi_j, \psi\})^n\wedge[S]\big)dv(S)\\
&\geq \lim_{r\ra0}\lim_{j\ra\infty}\int_{U_0} \big(\int_{\mc^{nm}}\chi_r(dd^c\max\{\varphi_j, \psi\})^n\wedge[S]\big)dv(S)\\
&=\lim_{r\ra0}\lim_{j\ra\infty}\int_{\{\|A\|<R_0\}}\Big(\int_{\mc^n}\chi_r(z,Az)(dd^c_z\max\{\varphi_j,A^*\psi\})^n\Big)d\tilde{v}(A)\\
&=\int_{\{\|A\|<R_0\}}\Big(\int_{\{0\}}(dd^c_z\max\{\varphi,A^*\psi\})^n\Big)d\tilde{v}(A).
\end{align*}
The second equation is derived from the monotone convergence
$$(dd^c\max\{\varphi_j,\psi\})^n\wedge(dd^c\log|(z,w)|)^m\ra(dd^c\max\{\varphi,\psi\})^n\wedge(dd^c\log|(z,w)|)^m.$$
Combining $$\sup_{j,A}\int_{r_0\mathbb{B}^n}(dd^c\max\{\varphi_j,A^*\psi\})^n<+\infty$$
and the weak convergence $$(dd^c\max\{\varphi_j,A^*\psi\})^n\ra(dd^c\max\{\varphi,A^*\psi\})^n,\ \forall A,$$
the dominated convergence theorem yields the last equation.
Since $R_0$ is arbitrary, the conclusion follows as desired.

%To prove Theorem \ref{thm:max for E} in the full generality, we need the following lemma

\begin{itemize}
\item
Step 2. Theorem \ref{thm:max final version} holds for every $A$ outside a pluripolar set.
\end{itemize}

We may assume $\varphi\in\mathcal{F}(\mathbb{B}^n)$, $\psi<-2$, $\psi$ is defined in a neighborhood of $\overline{\mathbb{B}^m}$ and work on $\mathbb{B}^{nm}=\{A:\|A\|<1\}$.
By \cite[Lemma 2.1, Proposition 5.1]{Ce04}, there is a decreasing sequence $\varphi_j\in\mbox{PSH}(\mathbb{B}^n)\cap \mathcal{C}(\overline{\mathbb{B}^n})$ so that $\lim_{j\ra\infty}\varphi_j=\varphi$, $\varphi_j|_{\partial \mathbb{B}^n} =0$ and $$\int_{\mathbb{B}^n}(dd^c\varphi)^n=\lim_{j\ra\infty}\int_{\mathbb{B}^n}(dd^c\varphi_j)^n=\sup_{j}\int_{\mathbb{B}^n}(dd^c\varphi_j)^n<+\infty.$$
For each $j$, there is a positive number $r_j<1$ such that $\varphi_j\geq-1$ on $\{|z|\geq r_j\}$,
we then choose a smooth, increasing function $\chi_j\in\mathcal{C}^{\infty}(\mathbb{R};\mr)$ such that $\chi_j(t)=t$ near $\{t\leq\log r_j\}$ and $\chi_j(t)=0$ near $\{t\geq0\}$.
Furthermore, $\chi_j$ can be chosen so that $\chi_j$ uniformly converges to $\chi$ on $\mr$, where $\chi(t)=t$ on $\{t<0\}$, $\chi(t)=0$ on $\{t\geq0\}$.
Consider the following plurisubharmonic functions on $\mathbb{B}^n\times\mathbb{B}^{nm}\times H$
\begin{align*}
\Phi_j:=\max\{\varphi_j,A^*\psi_j\},
\Phi:=\max\{\varphi,A^*\psi\},
\Psi:=\max\{\log|z|,\mathrm{Re}w\},
\end{align*}
where $\psi_j=\psi*\rho_{\varepsilon_j}$ is the standard approximation of $\psi$.
Then we define
\begin{align*}
(A,w)\mapsto V_j(A,w):&=-\int_{\mathrm{Re}w}^0\chi_j^{\prime}(t)\Big(\int_{\{z\in\mathbb{B}^n:\log|z|<t\}\}}(dd^c_z\Phi_j)^n\Big)dt\\
&=\int_{\mathbb{B}^n}\chi_j(\Psi)(dd^c_z\Phi_j)^n,
\end{align*}
\begin{align*}
(A,w)\mapsto V(A,w):&=-\int_{\mathrm{Re}w}^0\Big(\int_{\{z\in\mathbb{B}^n:\log|z|<t\}\}}(dd^c_z\Phi)^n\Big)dt\\
&=\int_{\mathbb{B}^n}\Psi(dd^c_z\Phi)^n.
\end{align*}

\begin{itemize}
\item
Step 2.1. $V_j$ is plurisubharmonic on $\mathbb{B}^{nm}\times H$.
\end{itemize}

To achieve this, we introduce the following approximating sequence:
\begin{align*}
V_{j,k}(A,w):=\int_{\mathbb{B}^n}\chi_j(\Psi_k)(dd^c_z\Phi_{j,k})^n,
\end{align*}
where $\Psi_k:=\max_{\varepsilon_k}\{\log|z|,\mathrm{Re}w\}$, $\Phi_{j,k}:=\max_{\varepsilon_k}\{\varphi_j*\rho_{\varepsilon_k},A^*\psi_j\}$.
Although $\Phi_{j,k}$ is not well-defined on the whole $\mathbb{B}^n$, $\chi_j(\Psi)$ has compact support in $\mathbb{B}^n$, the above integrations still make sense for $k$ large.
We would like to show
\begin{itemize}
\item[(1)]
$V_{j,k}$ is plurisubharmonic on $\mathbb{B}^{nm}\times H$,
\item[(2)]
$\lim_{k\ra\infty}V_{j,k}=V_j$ on $\mathbb{B}^{nm}\times H$.
\end{itemize}
As in the proof of Theorem \ref{thm:max},
for test form $h$ with compact support on $\mathbb{B}^{nm}\times H$, we compute
\begin{align*}
\langle dd^c_{A,w}V_{j,k},h\rangle
=&\langle V_{j,k},dd^c_{A,w}h\rangle\\
=&\int_{\mathbb{B}^n\times\mathbb{B}^{nm}\times H} dd^c\chi_j(\Psi_k)\wedge (dd^c\Phi_{j,k})^n\wedge h\\
=&\int_{\{|z|<r_j\}\times\mathbb{B}^{nm}\times H} dd^c\Psi_k\wedge(dd^c\Phi_{j,k})^n\wedge h.
\end{align*}
The last equality holds because
$$dd^c\chi_j(\Psi_k)\wedge(dd^c\Phi_{j,k})^n=dd^c\chi_j(\log|z|)\wedge (dd^c\varphi_j*\rho_{\varepsilon_k})^n=0$$ on $\{|z|\geq r_j\}\times\mathbb{B}^{nm}\times H.$
$(1)$ is therefore proved.
Since $(dd^c_z\Phi_{j,k})^n\ra(dd^c_z\Phi_j)^n$ weakly and $\chi_j(\Psi_k)$ converges uniformly to $\chi_j(\Psi)$, $(2)$ also follows.
We then infer that $dd^cV_j\geq0$ in the sense of currents.
To prove $V_j$ is plurisubharmonic, it is sufficient to show $V_j$ is continuous.
Indeed, for $(A_0,w_0)$ being fixed, since $\max\{\varphi_j,A^*\psi_j\}$ uniformly converges to $\max\{\varphi_j,A_0^*\psi_j\}$ as $A\ra A_0$, we obtain the weak convergence $(dd^c_z\Phi_j(z,A))^n\ra(dd^c_z\Phi_j(z,A_0))^n$.
It is clear that $\chi_j(\Psi(z,w))\rightrightarrows\chi_j(\Psi(z,w_0))$ on $\mathbb{B}^n$ as $ w\ra w_0$, the desired continuity follows.
\begin{itemize}
\item
Step 2.2. $\lim_{j\ra\infty} V_j=V$ on $\mathbb{B}^{nm}\times H$.
\end{itemize}

We fix $(A,w)\in \mathbb{B}^{nm}\times H$.
By $\Phi_j(z,A)= \varphi_j(z)$ near $\partial\mathbb{B}^n$, an integration by parts yields that $\int_{\mathbb{B}^n}(dd^c_z\Phi_j(z,A))^n=\int_{\mathbb{B}^n}(dd^c\varphi_j)^n$, hence is uniformly bounded.
Moreover, since $\chi_j$ converges uniformly, we see that for every $\delta>0$, if $j_0$ is large enough,
\begin{align*}
\Big|\limsup_{j\ra\infty}\big(\int_{\mathbb{B}^n}\chi_j(\Psi)(dd^c_z\Phi_j(z,A)))^n
-\int_{\mathbb{B}^n}\chi_{j_0}(\Psi)(dd^c_z\Phi_j(z,A))^n\big) \Big|<\delta.
\end{align*}
Similarly, the inequality
$$\int_{\mathbb{B}^n}(dd^c_z\Phi(z,A))^n=\lim_{j\ra\infty}\int_{\mathbb{B}^n}(dd^c_z\Phi_j(z,A))^n<+\infty$$
implies that for $j_0$ large,
$$|\int_{\mathbb{B}^n}\Psi(dd^c_z\Phi(z,A))^n-\int_{\mathbb{B}^n}\chi_{j_0}(\Psi)(dd^c_z\Phi(z,A))^n|<\delta.$$
From the weak convergence $(dd^c_z\Phi_j)^n\ra(dd^c_z\Phi)^n$, we can deduce
$$\lim_{j\ra\infty}\int_{\mathbb{B}^n}\chi_{j_0}(\Psi)(dd^c_z\Phi_j(z,A))^n=\int_{\mathbb{B}^n}\chi_{j_0}(\Psi)(dd^c_z\Phi(z,A))^n.$$
Letting $\delta\ra0$, it follows that
\begin{align*}
\lim_{j\ra\infty}V_j(A,w)=\lim_{j\ra\infty}\int_{\mathbb{B}^n}\chi_{j}(\Psi)(dd^c_z\Phi_j(z,A))^n=\int_{\mathbb{B}^n}\Psi(dd^c_z\Phi(z,A))^n=V(A,w).
\end{align*}
Since $V_j$ is independent of $\mathrm{Im}w$, for $r>0$, the function $U^r_j$ on $\mathbb{B}^{nm}$
\begin{align*}
A\mapsto  U_j^r(A):=\inf_{t<-2}\{V_j(A,t)-rt\}
\end{align*}
is plurisubharmonic.
Similarly, we define
\begin{align*}
A\mapsto U^r(A):=\inf_{t<-2}\{V(A,t)-rt\}
\end{align*}
\begin{itemize}
\item
Step 2.3. Get a uniform bound for $\{U^r_j(A):j=1,2,\cdots\}$ for $A$ generic.
\end{itemize}

We shall prove if $\int_{\{0\}}(dd^c_z\Phi(z,A))^n=0$, then $\liminf_{j\ra\infty}U_j^r(A)>-\infty$ for all $r>0$.
%First, we notice that $\Phi_j(z,Az)\geq\varphi_j(z)$, comparison principle gives us
%\begin{align*}
%\int_{\mathbb{B}^n}(dd^c_z\Phi_j)^n
%\leq \int_{\mathbb{B}^n}(dd^c\varphi_j)^n
%\end{align*}
%and therefore $\{\int_{\mathbb{B}^n}(dd^c_z\Phi_j(z,Az))^n:j\in\mathbb{N}^*,A\in\mathbb{B}^{nm}\}$ is uniformly bounded.
Indeed, the weak convergence $(dd^c_z\Phi_j)^n\ra(dd^c_z\Phi)^n$ implies there is a number $t_0<-2$ so that $\int_{\{\log|z|<t_0\}}(dd^c_z\Phi_j)^n\leq\frac{r}{2}$ for all $j$.
So we deduce that the infimum of $t\mapsto V_j(A,t)-rt$ is taken on the segment $(t_0,-2)$ for all $j$.
However,
$$V_j(A,t)=\int_{\mathbb{B}^n}\chi_j(\Psi)(dd^c_z\Phi_j)^n\geq t_0\int_{\mathbb{B}^n}(dd^c_z\Phi_j)^n,$$
and the latter is uniformly bounded thanks to
$$\sup_{j}\int_{\mathbb{B}^n}(dd^c_z\Phi_j)^n=\sup_{j}\int_{\mathbb{B}^n}(dd^c\varphi_j)^n<+\infty.$$
\begin{itemize}
\item
Step 2.4. $E_0:=\{A\in\mathbb{B}^{nm}:\int_{\{0\}}(dd^c_z\Phi(z,Az))^n>0\}$ is pluripolar.
\end{itemize}

Since $U^r_j$ does not converge to $-\infty$, we obtain a plurisubharmonic function on $\mathbb{B}^{nm}$
\begin{align*}
A\mapsto\tilde{U}^r(A):=\lim_{j\ra\infty}(\sup_{k\geq j}U^r_{j}(A))^*.
\end{align*}
Using the fact that $\lim_{j\ra\infty} V_j=V$ on $\mathbb{B}^{nm}\times H$, we get $\limsup_{j\ra\infty} U_{j}^r\leq U^r$ on $\mathbb{B}^{nm}$.
Therefore $\tilde{U}^r\leq U^r$ outside a pluripolar set $P_r$.
This is because $$P_{r,j}:=\{A:(\sup_{k\geq j}U^{r}_{j})^*\neq \sup_{k\geq j}U^{r}_{j}\}$$ is pluripolar for every $j$, then $P_r:=\cup_j P_{r,j}$ is what we want.
We observe that
\begin{itemize}
\item
$U^r(A)=-\infty$, if $\int_{\{0\}}(dd^c_z\Phi(z,A))^n>r$,
\item
$U^r(A)>-\infty$, if $\int_{\{0\}}(dd^c_z\Phi(z,A))^n<r$.
\end{itemize}
For $r\geq 0$, we set
\begin{align*}
E_r:=\{A\in\mathbb{B}^{nm}:\int_{\{0\}}(dd^c_z\Phi(z,A))^n>r\}.
\end{align*}
The above analyses yield that for $r>0$,
\begin{itemize}
\item
$E_r\subseteq\{U^r=-\infty\}$,
\item
$\{U^r=-\infty\}\subseteq\{\tilde{U}^r=-\infty\}\cup P_r$.
\end{itemize}
Note that $$E_0=\bigcup_{k=1}^{\infty}\{U^{\frac{1}{k}}=-\infty\}\subseteq\bigcup_{k=1}^{\infty} (\{\tilde{U}^{\frac{1}{k}}=-\infty\}\cup P_{\frac{1}{k}}),$$
the proof is therefore complete.
\end{proof}

Finally we can deduce Theorem \ref{thm:n-1} from Theorem \ref{thm:max final version} using Lemma \ref{lem:Blocki formula},
in the same way as in the case that $\varphi, \phi$ has isolated singularities, as shown in \S \ref{sec:max isolated sing}.

	\end{document}